\documentclass{amsart}
\usepackage[T1]{fontenc}
\usepackage[utf8]{inputenc}

\usepackage[paperwidth=182mm,paperheight=257mm,text={140mm,210mm},centering]{geometry}
\usepackage{tikz-cd,microtype,amssymb,lmodern,mathrsfs}

\usepackage{hyperref}

\usepackage{enumitem}
\setlist[enumerate]{labelindent=\parindent,topsep=0.4ex,itemsep=0.1ex}
\setlist[itemize]{labelindent=\parindent,topsep=0.4ex,itemsep=0.1ex}
\setlist[enumerate,1]{labelindent=\parindent, leftmargin=*,label=\textup{(\arabic*)},ref=\text{\arabic*}}
\setlist[enumerate,2]{labelindent=\parindent, leftmargin=*,label=\textup{(\alph*)},ref=\text{\alph*}}

\numberwithin{equation}{section}
\theoremstyle{plain}
\newtheorem{theorem}[equation]{Theorem}
\newtheorem{lemma}[equation]{Lemma}
\newtheorem*{theorem*}{Theorem}
\newtheorem*{claim}{Claim}
\newtheorem{corollary}[equation]{Corollary}
\theoremstyle{definition}
\newtheorem{construction}[equation]{Construction}
\newtheorem{situation}[equation]{}
\newtheorem{definition}[equation]{Definition}
\newtheorem*{definition*}{Definition}
\newtheorem*{questionA}{Question A}
\newtheorem*{questionB}{Question B}
\newtheorem{example}[equation]{Example}
\newtheorem{notation}[equation]{Notation}
\newtheorem{remark}[equation]{Remark}
\newtheorem*{remark*}{Remarks}

\tikzcdset{arrow style=tikz}
\usetikzlibrary{decorations.pathmorphing}

\renewcommand{\mathbb}{\mathbf}
\renewcommand{\setminus}{\mathbin{\rule[0.2em]{0.67em}{0.12em}}}%

\title[Visibility and divisibility]%
{Revisiting Dwork cohomology: Visibility and divisibility of Frobenius eigenvalues in rigid cohomology}

\author{Daqing Wan}
\address{Department of Mathematics, University of California, Irvine, CA 92697-3875 USA.}
\email{dwan@math.uci.edu}

\author{Dingxin Zhang}
\address{C629 Shuangqing Complex Building A, Tsinghua University, Beijing~100086 China.}
\email{dingxin@tsinghua.edu.cn}

\begin{document}
\maketitle

\begin{abstract}
In this paper, we investigate Frobenius eigenvalues of the compactly
supported rigid cohomology of a variety defined over a finite field
with $q$ elements, using Dwork's method.  Our study yields several
arithmetic consequences.  Firstly, we establish that the zeta
functions of a set of related affine varieties can reveal all
Frobenius eigenvalues of the rigid cohomology of the variety up to a
Tate twist.  This result does not seem to be known for $\ell$-adic
cohomology.  As a second application, we provide several
$q$-divisibility lower bounds for the Frobenius eigenvalues of the
rigid cohomology of the variety, in terms of the dimension and
multi-degrees of the defining equations.  These divisibility bounds
for rigid cohomology are generally better than what is suggested from
the best known divisibility bounds in $\ell$-adic cohomology, both
before and after the middle cohomological degree.
\end{abstract}

\setcounter{tocdepth}{1}
\tableofcontents

\section{Introduction}
\label{sec:intro}

Dwork \cite{dwork:zeta-function-hypersurface-1} engineered a
cohomology theory in order to study the zeta function of a projective
hypersurface over a finite field $\mathbb{F}_{q}$ of $q$ elements with
characteristic $p$.  In this paper, we revisit his construction, and
study some problems on Frobenius eigenvalues of affine varieties in
rigid cohomology.

\subsection{Visibility of Frobenius eigenvalues}
Our first theorem is about the \emph{visibility} of Frobenius
eigenvalues in the zeta function of an affine variety.

Let \(\mathbb{F}_{q}\) be a finite field with \(q\) elements and
characteristic \(p\).  Given an algebraic variety \(Z\) over
\(\mathbb{F}_{q}\), its \emph{zeta function} is defined as
\begin{equation*}
  \zeta_Z(t) = \exp\left\{ \sum_{m=1}^{\infty}  \frac{|Z(\mathbb{F}_{q^m})|}{m}t^{m}\right\}.
\end{equation*}
Weil conjectured, Dwork subsequently
proved~\cite{dwork:rationality-zeta-function}, that \(\zeta_Z(t)\) is
a rational function.

By trace formulae in rigid cohomology
(due to Étesse and Le~Stum~\cite{etesse-le-stum:l-functions-associated-to-overconvergent-f-isocrystals-1})
and \(\ell\)-adic cohomology (see, e.g., \cite[Rapport, \S\S4--6]{deligne:sga4.5}),
the zeta function is an alternating product
\begin{equation}\label{eq:grothendieck-trace-formula}
  \zeta_Z(t) = \prod_{i=0}^{2\dim Z} \det\left(1 - t\cdot \mathrm{Frob}_q | \mathrm{H}^{i}_{c}(Z)\right)^{(-1)^{i+1}}.
\end{equation}
Here, \(\mathrm{H}^{i}_c(Z)\) could mean either Berthelot's compactly supported
rigid cohomology \(\mathrm{H}^i_{\mathrm{rig},c}(Z)\),
or compactly supported \(\ell\)-adic cohomology \(\mathrm{H}^{i}_{c}(Z_{\overline{\mathbb{F}}_q},\mathbb{Q}_{\ell})\)
(\(\overline{\mathbb{F}}_{q}\) is a fixed algebraic closure of
\(\mathbb{F}_q\),
\(Z_{\overline{\mathbb{F}}_q}=Z\otimes_{\mathbb{F}_q}\overline{\mathbb{F}}_q\),
\(\ell\neq p\) is a prime number).  Note that the finite
dimensionality of \(\mathrm{H}^{i}_c(Z)\) also yields a cohomological
proof for the rationality of the zeta function.

By the trace formulae, reciprocal roots and poles of \(\zeta_{Z}(t)\) constitute
a subset of the Frobenius eigenvalues of
\(\mathrm{H}^{\ast}_{c}(Z)\).  When \(Z\) is
\emph{smooth} and \emph{proper} over \(\mathbb{F}_{q}\), the converse
is also true, as a result of Deligne's resolution of Weil's
conjecture~\cite[Théorème~1.6]{deligne:weil-1} (for rigid cohomology,
see Katz--Messing~\cite{katz-messing:consequence-riemann-hypothesis}).
In such cases, the Frobenius eigenvalues of \(\mathrm{H}^{i}_{c}(Z)\)
are algebraic integers having archimedean absolute value \(q^{i/2}\) (with
respect to any abstract embedding
\(W(\mathbb{F}_{q})[1/p]\hookrightarrow \mathbb{C}\) or
\(\mathbb{Q}_{\ell}\hookrightarrow \mathbb{C}\)).  Therefore the
denominator and numerator of the right hand side of
\eqref{eq:grothendieck-trace-formula} do not have common factors; the
zeta function alone can recover the Frobenius eigenvalues.

Without the smooth proper condition, the linear factors of the
determinants in \eqref{eq:grothendieck-trace-formula} could cancel
out.  If a cancellation happens, \(\zeta_Z(t)\) may not be capable of
witnessing all the Frobenius eigenvalues, not even up to Tate twist.
Here is a simple example.  Let \(X\) be a general nonsingular cubic
curve in \(\mathbb{A}^{2}_{\mathbb{F}_q}\), and
\(Y = \mathbb{A}^2_{\mathbb{F}_q} \setminus X\) its complement.  Then
the zeta function of the affine variety \(Z=X \sqcup Y\) equals that
of \(\mathbb{A}^2_{\mathbb{F}_q}\), namely \((1-q^{2}t)^{-1}\). On the
other hand, there exist Frobenius eigenvalues of
\(\mathrm{H}^1_{c}(Z) = \mathrm{H}^1_{c}(X) \oplus \mathrm{H}^1_{c}(Y)\) of
absolute value \(q^{1/2}\).

Our first theorem asserts that, if we are willing to take the
\emph{defining equations} of an affine variety \(Z\) into the consideration,
then we can
\emph{recover all the Frobenius eigenvalues of \(Z\) up to Tate twist}
from zeta functions of finitely many varieties related to \(Z\).
In order to give the precise statement,
let us introduce some terminologies.

\begin{definition}
Let \(\Gamma=\{\Gamma_{a}(t), \Gamma_b(t), \ldots\} \subset 1 + t\mathbb{C}_p[\![t]\!]\)
be a collection of \(p\)-adic meromorphic function on \(\mathbb{C}_{p}\)
(i.e., fractions of \(p\)-adic entire functions).
We say a \(p\)-adic number \(\gamma \in \mathbb{C}_p\setminus\{0\}\) is \emph{visible}
in \(\Gamma\), if \(\Gamma_a(\gamma^{-1})=0\) or \(\infty\) for some \(\Gamma_a \in \Gamma\).
We say \(\gamma\) is \emph{weakly visible} in \(\Gamma\), if there exists
\(m \in \mathbb{Z}\) such that \(q^{m}\gamma\) is visible in \(\Gamma\).
\end{definition}

Now let
\(f_{1},\ldots,f_{r} \in \mathbb{F}_{q}[x_1,\ldots,x_n]\) be a collection
of polynomials. For every subset \(I \subset \{1,2,\ldots,r\}\),
set \(Z_{I}= \operatorname{Spec}\mathbb{F}_q[x_1,\ldots,x_n]/(f_i:i\in I) \subset \mathbb{A}^{n}_{\mathbb{F}_q}\),
and \(Z_{I}^{\ast} = Z_{I} \cap \mathbb{G}_{\mathrm{m}}^{n}\).
Write \(Z=Z_{\{1,2,\ldots,r\}}\).

\begin{theorem}%
\label{theorem:cancellation}
Let \(Z\) be the vanishing locus of
\(f_{1},\ldots,f_{r} \in \mathbb{F}_{q}[x_{1},\ldots,x_{n}]\) in
\(\mathbb{A}^{n}\).  Then any Frobenius eigenvalue of
\(\mathrm{H}^{\bullet}_{\mathrm{rig},c}(Z)\) is weakly visible in the
finite set
\(\{\zeta_{Z_{I}^{\ast}}(t): {I\subset\{1,2,\ldots,r\}}\}\).
\end{theorem}

The theorem is already interesting when
\(Z\subset\mathbb{A}^n_{\mathbb{F}_q}\) is an affine hypersurface.  In
this special situation, it asserts that if a Frobenius eigenvalue
\(\lambda\) is canceled out in the zeta function, then either it
equals \(q^m\) for some \(m\), or there must exist a reciprocal root
or reciprocal pole of \(\zeta_{Z\cap \mathbb{G}_{\mathrm{m}}^n}(t)\)
that equals a Tate twist of \(\lambda\).  Thus the zeta function
\(\zeta_{Z\cap\mathbb{G}_{\mathrm{m}}^n}\) alone can recover all the
Frobenius eigenvalues of \(\mathrm{H}^{\bullet}_{\mathrm{rig},c}(Z)\)
up to Tate twist.

In view of the ``motivic'' philosophy, the same result should
also hold for \(\ell\)-adic cohomology. But our method---\(p\)-adic in
nature---depends upon an explicit chain model of rigid cohomology, and
does not work in the \(\ell\)-adic context.

\subsection{Divisibility of Frobenius eigenvalues}
The second result of this paper concerns the notion of
$q$-\emph{divisibility} as algebraic integers of Frobenius eigenvalues
of affine and projective varieties.

Let \(f_1,\ldots,f_r \in \mathbb{F}_{q}[x_1,\ldots,x_n]\) be a collection of
polynomials. Write \(d_j = \deg f_j\). Without loss of generality, we will assume that all the degrees $d_j$ are
positive. By rearranging we could and will assume
\(d_1\geq d_2\geq \cdots \geq d_r\).
Let \[Z = \operatorname{Spec}\mathbb{F}_q[x_1,\ldots,x_n]/(f_1,\ldots,f_r)\]
be the vanishing scheme of these polynomials in \(\mathbb{A}^{n}_{\mathbb{F}_q}\). For any integer $j\geq 0$,
we define a  non-negative integer
\begin{equation*}
  \mu_j(n;d_1,\cdots, d_r) = j + \max\left\{0, \left\lceil \frac{n-j-\sum_{i=1}^{r}d_i}{d_1}\right\rceil\right\}.
\end{equation*}

Recall that the classical Ax--Katz
theorem~\cite{ax:zeroes-of-polynomials-over-finite-fields,katz:on-a-theorem-of-ax} states that all the
reciprocal roots and poles of the zeta function of \(Z\) are divisible by
\(q^{\mu_0(n;d_1,\ldots,d_r)}\) as algebraic integers. This divisibility was
later upgraded to a divisibility on Frobenius eigenvalues on \(\ell\)-adic
cohomology: Esnault and Katz \cite{esnault-katz:cohomological-divisibility}
showed that the Frobenius eigenvalues of
\(\mathrm{H}^{\bullet}_{c}(Z_{\overline{\mathbb{F}}_q},\mathbb{Q}_{\ell})\)
are divisible by \(q^{\mu_{0}(n;d_1,\ldots,d_r)}\), furthermore,  for $j\geq 0$,
the Frobenius eigenvalues of \(\mathrm{H}^{n-1+j}_{c}(Z_{\overline{\mathbb{F}}_q},\mathbb{Q}_{\ell})\)
are divisible by \(q^{\mu_{j}(n;d_1,\ldots,d_r)}\).

The theorem of Esnault and Katz does not give the most optimal bound when \(r>1\).
Recently, Esnault and the first author~\cite{esnault-wan:divisibility} revisited
this theme. Based on their study for projective varieties, they suggested
a divisibility bound beyond the middle cohomological degree better than
the Esnault--Katz bound.

\begin{questionA}
Is it true that any Frobenius eigenvalue of
\(\mathrm{H}^{\dim Z+j}_{c}(Z_{\overline{\mathbb{F}}_q},\mathbb{Q}_{\ell})\)
is divisible by \(q^{\mu_j(n;d_1,\ldots,d_r)}\), in the ring of algebraic integers, for all
integers \(j\) satisfying $0\leq j\leq \dim Z$?
\end{questionA}

This question, if answered affirmatively, would simultaneously improve the results of
\cite{esnault-katz:cohomological-divisibility} and Deligne's integrality
theorem \cite[Exposé XXI,~\S5]{sga7-2} beyond the middle cohomological degree.

Since Frobenius eigenvalues are supposed to be ``motivic'', one is led to ask
the same question for the Frobenius eigenvalues of rigid cohomology.

\begin{questionB}
Is it true that any Frobenius eigenvalue of
\(\mathrm{H}^{\dim Z+j}_{\mathrm{rig},c}(Z)\)
is divisible by \(q^{\mu_j(n;d_1,\ldots,d_r)}\), in the ring of algebraic integers, for all
integers \(j\) satisfying $0\leq j\leq \dim Z$?
\end{questionB}

In this article, we show that:

\begin{theorem*}
  Question B has an affirmative answer.
\end{theorem*}

Unlike in the \(\ell\)-adic situation, where theorems are usually proved via dévissage,
we shall pursue the divisibility for rigid cohomology using a different
method, via Dwork's $p$-adic theory, refining and upgrading the chain level
approach in \cite{wan:poles} to Dwork cohomology, and then by comparison with
rigid cohomology.

Somewhat surprisingly, the bounds we obtain through the
\(p\)-adic methods are sharper than anticipated by
Question~B. Also, our approach gives divisibility bounds
before middle cohomological degree, improving the Ax--Katz type
bounds of~\cite{esnault-katz:cohomological-divisibility}.

To state our bounds beyond middle cohomological degree, define
\begin{equation*}
  d_i^* =
  \begin{cases}
    d_i, & \text{ if }1\leq i \leq n-\dim Z; \\
    1, & \text{ if }i>n - \dim Z, \text{ and }d_i = d_1; \\
    0,& \text{ if }i> n-\dim Z, \text{ and }d_i < d_1.
  \end{cases}
\end{equation*}
For integer $j\geq 0$, define another non-negative integer
\begin{equation*}
  \nu_{j}(n;d_1,\ldots,d_r) = j + \max\left\{ 0, \left\lceil \frac{n-j-\sum_{i=1}^{r}d_{i}^{*}}{d_1}\right\rceil \right\}.
\end{equation*}
Note that \(d_{i} \geq d_{i}^{*}\), thus
\(\nu_j(n;d_1,\ldots,d_r) \geq \mu_j(n;d_1,\ldots,d_r)\);
also the numbers \(\nu_j(n;d_1,\ldots,d_r)\) form an increasing sequence in $j$.
These numbers depend on the degrees of the defining equations of \(Z\), and also on the dimension of \(Z\).

If \(Z\) is a complete intersection by \(f_{1},\ldots,f_{r}\),
namely, if $n-\dim Z=r$,
one checks that $d_i^*=d_i$ and \(\nu_j(n;d_1,\ldots,d_r) = \mu_j(n;d_1,\ldots,d_r)\).

\begin{theorem}[Divisibility beyond middle cohomological degree]
  \label{theorem:beyond-middle-dimension}
  Let notation be as above.
  For every \(0\leq j\leq \dim Z\),
  \begin{itemize}
  \item   the Frobenius eigenvalues of
    \(\mathrm{H}^{\dim Z+j}_{\mathrm{rig},c}(Z)\)  are divisible by
    \(q^{\nu_{j}(n;d_1,\ldots,d_r)}\) in the ring of algebraic integers;
  \item   the Frobenius eigenvalues of
    \(\mathrm{H}^{\dim Z+1+j}_{\mathrm{rig},c}(\mathbb{A}^n_{\mathbb{F}_q}\setminus Z)\)
    are divisible by
    \(q^{\nu_{j}(n;d_1,\ldots,d_r)}\) in the ring of algebraic integers.
  \end{itemize}
\end{theorem}

\begin{remark}
  (a) The second item is the consequence of the first, thanks to the long
  exact sequence for compactly supported cohomology.

  (b) For any separated variety \(Z\) over \(\mathbb{F}_q\), the
  Frobenius eigenvalues of \(\mathrm{H}^{\ast}_{\mathrm{rig},c}(Z)\) are
  always algebraic integers. When \(Z\) is smooth proper, we use the Weil conjecture
  and the integrality of the zeta function, see
  \cite[Theorem~1]{katz-messing:consequence-riemann-hypothesis}. If \(Z\) is
  proper but possibly singular, we can produce a proper hypercovering using
  smooth proper varieties by alteration~\cite{de-jong:alteration}, and then
  apply cohomological
  descent~\cite{tsuzuki:cohomological-descent-rigid-cohomology-for-proper-coverings}. If \(Z\) is not
  proper, we can embed \(Z\) into a proper variety \(\overline{Z}\), and
  conclude by using the assertion for proper varieties and the long exact
  sequence
  \begin{equation*}
    \cdots \to \mathrm{H}^{i}_{\mathrm{rig},c}(Z) \to \mathrm{H}^{i}_{\mathrm{rig}}(\overline{Z}) \to \mathrm{H}^{i}_{\mathrm{rig}}(\overline{Z}\setminus Z) \to \cdots.
  \end{equation*}
  Our method is capable of seeing this too, see p.~\pageref{bootstrap}.
\end{remark}

Since \(\nu_j(n;d_1,\ldots,d_r) \geq \mu_j(n;d_1,\ldots,d_r)\),
Theorem~\ref{theorem:beyond-middle-dimension} establishes
an enhanced positive answer to Question B.
In the complete intersection case, the two bounds are identical.
If $Z$ is not a complete intersection, the divisibility bound in
Theorem~\ref{theorem:beyond-middle-dimension} can be strictly better.

What about before the middle cohomological degree?
In this range, the only known divisibility for $\ell$-adic  cohomology
is the theorem of Esnault--Katz which says that
the divisibility is by \(q^{\mu_0(n;d_1,\ldots,d_r)} \). We have an improved $p$-adic companion
in this range as well.

Since \(Z\) is cut out by \(r\) equations,
\(\mathrm{H}^{i}_{\mathrm{rig},c}(Z)\) and
\(\mathrm{H}^{i}_{c}(Z_{\overline{\mathbb{F}}_q},\mathbb{Q}_{\ell})\)
all vanish if \(i<n-r\) (see Lemma~\ref{lemma:vanishing-by-number-equations}).
So, we will assume that $i\geq n-r$.
If \(n-r=\dim Z\), i.e., \(Z\) is a complete
intersection, Theorem~\ref{theorem:beyond-middle-dimension}
already covers all the possible cohomological degrees \(i\) such that
\(\mathrm{H}^{i}_{\mathrm{rig},c}(Z)\) is nontrivial.
However, if \(Z\) is not a complete intersection,
\(\mathrm{H}^{i}_{\mathrm{rig},c}(Z)\) and
\(\mathrm{H}^{i}_{c}(Z_{\overline{\mathbb{F}}_q},\mathbb{Q}_{\ell})\)
could be nonzero for \(n-r\leq i < \dim Z\).
A novelty of our approach is that we can provide improved
divisibility information of Frobenius eigenvalues in these degrees as well, of course, only for rigid cohomology.

\begin{theorem}[Divisibility before middle cohomological degree]
  \label{theorem:before-middle-dimension}
  Let notation be as above. For every \(0\leq m \leq \dim Z - (n-r)\),
  the Frobenius eigenvalues of \(\mathrm{H}^{n-r+m}_{\mathrm{rig},c}(Z)\)
  are divisible by \(q^{\epsilon_m(n;d_1,\ldots,d_r)}\) in the ring of algebraic
  integers, where
  \begin{equation*}
    \epsilon_m(n;d_1,\ldots,d_r) =
    \max\left\{ 0, \left\lceil \frac{n-(d_1+\cdots + d_{r-m}+d_{r-m+1}^*+\cdots + d_{r}^{*})}{d_1}\right\rceil \right\}.
  \end{equation*}
  The Frobenius eigenvalues of
  \(\mathrm{H}^{n-r+1+m}_{\mathrm{rig},c}(\mathbb{A}^n_{\mathbb{F}_q}\setminus Z)\) are divisible
  by \(q^{\epsilon_m(n;d_1,\ldots,d_r)}\) in the ring of algebraic integers as well.
\end{theorem}

The numbers \(\epsilon_{m}(n;d_1,\ldots,d_r)\) (\(m=0,1,\ldots, \dim Z-(n-r)\)) form an increasing
sequence in the closed interval \([\mu_0(n;d_1,\ldots,d_r),\nu_0(n;d_1,\ldots,d_r)]\),
the smallest one \(\epsilon_0(n;d_1,\ldots,d_r)=\mu_0(n;d_1,\ldots,d_r)\)
being responsible for the Ax--Katz theorem and the Esnault--Katz theorem; and we have
\[\epsilon_{\dim Z-(n-r)}(n;d_1,\ldots,d_r)=\nu_0(n;d_1,\ldots,d_r).\]

\medskip
Theorems~\ref{theorem:beyond-middle-dimension},~\ref{theorem:before-middle-dimension}
have projective analogues.
For a closed subvariety \(Z\) of \(\mathbb{P}^n_{\mathbb{F}_q}\), set
\[
\mathrm{H}^{\ast}_{\mathrm{rig}}(Z)_{\mathrm{prim}}=
\operatorname{Coker}(\mathrm{H}^{\ast}_{\mathrm{rig}}(\mathbb{P}_{\mathbb{F}_q}^n) \to \mathrm{H}^{\ast}_{\mathrm{rig}}(Z)).
\]

\begin{theorem}%
  \label{theorem:projective-bound}
  Let \(f_1,\ldots,f_r \in \mathbb{F}_{q}[x_0,\ldots,x_n]\)
  be homogeneous polynomials of positive degrees \(d_1 \geq \cdots \geq d_r\). Let \(Z\) be the vanishing scheme of \(f_1,\ldots,f_r\)
  in \(\mathbb{P}^n_{\mathbb{F}_q}\).
  Then, for \(0\leq j\leq \dim Z\),
  \begin{itemize}
  \item the Frobenius eigenvalues of \(\mathrm{H}^{\dim Z+j}_{\mathrm{rig}}(Z)_{\mathrm{prim}}\)
    are divisible, as algebraic integers, by \(q^{\nu_j(n+1;d_1,\ldots,d_r)}\);
  \item the Frobenius eigenvalues of
    \(\mathrm{H}^{\dim Z+1+j}_{\mathrm{rig},c}(\mathbb{P}^n_{\mathbb{F}_q}\setminus Z)\) are divisible, as
    algebraic integers, by \(q^{\nu_{j}(n+1;d_1,\ldots,d_r)}\).
  \end{itemize}
  For \(0 \leq m \leq  \dim Z - (n-r)\),
  \begin{itemize}
  \item the Frobenius eigenvalues of \(\mathrm{H}^{n-r+m}_{\mathrm{rig}}(Z)_{\mathrm{prim}}\)
    are divisible, as algebraic integers, by \(q^{\epsilon_{m}(n+1;d_1,\ldots,d_r)}\);
  \item the Frobenius eigenvalues of
    \(\mathrm{H}^{n-r+1+m}_{\mathrm{rig},c}(\mathbb{P}^n_{\mathbb{F}_q}\setminus Z)\) are divisible, as
    algebraic integers, by \(q^{\epsilon_{m}(n+1;d_1,\ldots,d_r)}\).

  \end{itemize}
\end{theorem}


\begin{remark}
These divisibility theorems in rigid cohomology now raise new questions for \(\ell\)-adic cohomology and Hodge
theory, through the ``motivic'' philosophy.
\begin{itemize}
\item The Frobenius eigenvalues of \(\ell\)-adic cohomology groups of
affine or projective \(Z\) should also satisfy the divisibility stated
in Theorems~\ref{theorem:beyond-middle-dimension},
\ref{theorem:before-middle-dimension} and
\ref{theorem:projective-bound}.
\item For affine or projective varieties defined over the field
\(\mathbb{C}\) of complex numbers, the numbers
\[
  \nu_{j}(n;d_1,\ldots,d_r) \text{ and } \epsilon_m(n;d_1,\ldots,d_r)
\]
should give lower bounds for Hodge levels of compactly supported
singular cohomology of complex varieties cut out by a set of
polynomial equations of degrees \(d_1,\ldots,d_r\).
\end{itemize}
Our method, purely analytic, is certainly not applicable for \(\ell\)-adic
cohomology. But the chain level considerations may be useful for the
problem on Hodge levels.
\end{remark}

The paper is organized as follows.
Section~\ref{sec:exponential-model} introduces the overconvergent
Dwork cohomology, and provides the construction of a specific chain
model for computing it.  It also contains the statement of a
theorem of Baldassarri and Berthelot
\cite{baldassarri-berthelot:dwork-cohomology-for-singular-hypersurfaces}
which allows us to relate the overconvergent Dwork cohomology
with the compactly supported rigid cohomology of an affine variety.
Section~\ref{sec:visibility} proves
Theorem~\ref{theorem:cancellation}.  Section~\ref{sec:artin-hasse}
explains how to transplant the theory of Adolphson and Sperber,
which uses a more complicated model of Dwork cohomology, to
the situation that concerns us.  After proving two lemmas in
Section~\ref{sec:algebra}, we provide the proofs of
Theorems~\ref{theorem:beyond-middle-dimension}
and~\ref{theorem:before-middle-dimension} in
Section~\ref{sec:proof-bound}.

\subsection*{Notation and conventions}

Throughout the main body of the paper, \(\mathbb{N} = \{0,1,\ldots\}\)
denotes the set of \emph{non-negative} integers.  We fix a prime
number \(p\), and let \(q\) be a power of \(p\).  Let
\(\mathcal{O}_{K} = W(\mathbb{F}_{q})[\zeta_{p}]\), where
\(\zeta_{p} \neq 1\) is a \(p\textsuperscript{th}\) root of unity.
Note that \(\mathcal{O}_{K}\) contains an element \(\pi\) satisfying
\(\pi^{p-1}+p=0\).  Let \(K\) be the field of fractions of
\(\mathcal{O}_{K}\).  Let \(|\cdot|\) be the ultrametric on \(K\)
extending that of \(\mathbb{Q}_{p}\).  The \(p\)-power Frobenius map
of \(\mathbb{F}_{q}\) lifts to an automorphism
\(\tau \in \mathrm{Gal}(K/\mathbb{Q}_{p}(\zeta_{p}))\) such that
\(\tau(\zeta_{p})=\zeta_{p}\) (thus \(\tau(\pi)=\pi\)), and
\(\tau^{q}=\mathrm{Id}\).

We consider exclusively rigid cohomology over the base field \(K\).
Thus, for an algebraic variety \(X\) over \(\mathbb{F}_{q}\), and an
overconvergent F-isocrystal \(\mathcal{E}\) on \(X\), the rigid
cohomology groups \(\mathrm{H}^{i}_{\mathrm{rig,c}}(X;\mathcal{E})\)
and \(\mathrm{H}^{i}_{\mathrm{rig}}(X;\mathcal{E})\), are all finite
dimensional \(K\)-vector spaces (by
Kedlaya~\cite{kedlaya:finiteness-rigid-cohomology}).  When
\(\mathcal{E}=\mathcal{O}_{X}\) is the constant isocrystal, we will
simply write \(\mathrm{H}^{i}_{\mathrm{rig}}(X)\) or
\(\mathrm{H}^{i}_{\mathrm{rig,c}}(X)\) for its rigid cohomology.

\subsection*{Acknowledgment}
We are grateful to Steve Sperber for useful discussions and for
sharing Tsuzuki's proof of Theorem~\ref{theorem:comparison} with us.
We also thank Yichen Qin for pointing out many inaccuracies in the
paper.

\section{Overconvergent Dwork cohomology}
\label{sec:exponential-model}
After establishing the rationality of the zeta function, Dwork proceeded
to pioneer the study of \(p\)-adic absolute values of the reciprocal roots and
zeros of zeta functions of an algebraic variety \(Z\) over \(\mathbb{F}_{q}\).
He proved the famous ``Newton above Hodge'' theorem when \(Z\)
is a nonsingular hypersurfaces in a projective
space~\cite{dwork:zeta-function-hypersurface-1}.
He accomplished this by devising some explicit chain complexes of
\(p\)-adic Banach spaces, which are equipped with some chain-level
representations of the Frobenius operation.  A substantial portion of
his work was centered around the chain-level investigations.  He only
delved into the cohomology spaces under special occasions where finite
dimensionality can be shown.  But even without knowing finiteness of
the cohomology spaces, the chain level approach can still extract
numerous properties of the zeta function.

Actually, Dwork didn't just design one chain complex, but rather
infinitely many different chain complexes, each corresponds to a
specific ``splitting function'' in his terminology.  The
chain-level Frobenius operators associated with these chain models are
all capable of calculating the zeta function of a variety, regardless
of whether the variety is nonsingular or not.  In some applications,
such as the ``visibility theorem''~\ref{theorem:cancellation},
knowledge of some formal properties of these operators suffices.  For
more intricate analysis, such as many important results of Adolphson
and
Sperber~\cite{adolphson-sperber:chevalley-warning,adolphson-sperber:newton-polyhedra-degree-l-function,adolphson-sperber:exponential-sums-newton-polyhedra},
and the ``divisibility theorem'' that relies on them, some
specific chain model needed to be used.

Since our first goal is to prove the visibility theorem, in this
section we will focus exclusively on a most straightforward
(overconvergent) Dwork complex.  In effect, it is the overconvergent
de~Rham complex of an exponentially twisted integrable connection on a
\(p\)-adic polydisk.  We shall state a theorem of Baldassarri and Berthelot which relates the
overconvergent Dwork cohomology of an affine variety \(Z\) with the
rigid cohomology of \(Z\) with compact support.

The role of a subtler model, the \emph{Artin--Hasse model}, will be
clarified in a subsequent section, before we prove the divisibility
theorem.

\begin{definition}\label{definition:dwork-crystal}
On the structure sheaf of the rigid analytic affine line
\(\mathbb{A}^{1,\mathrm{an}}_K\) (with coordinate \(z\)), we define an
integrable connection \(\nabla_{\pi}\) by the formula
\begin{equation*}
\nabla_{\pi}\xi = \mathrm{d}\xi + \pi \xi \mathrm{d}z.
\end{equation*}
The connection is overconvergent.  It is equipped with a Frobenius
structure
\begin{equation*}
\varphi_{\pi}\colon \xi \mapsto \xi^{\tau}(z^p) \cdot \theta(z)^{-1}
\end{equation*}
where \(\theta(z)\) is the Dwork exponential
\begin{equation*}
\theta(z) = \exp(\pi z - \pi z^{p}).
\end{equation*}
It is well-known that the radius of convergence of \(\theta(z)\) is
\(|p|^{-\frac{p-1}{p^2}}\), thus overconvergent
(cf.~\cite[p.~396, Theorem (b)]{robert:p-adic-analysis}).
The pair \((\nabla_{\pi}, \varphi_{\pi})\) defines an overconvergent
F-isocrystal \(\mathcal{L}_{\pi}\) on \(\mathbb{A}^{1}_{k}\), called
the \emph{Dwork crystal}.  See~\cite[\S4.2.1 and
\S8.3]{le-stum:rigid-cohomology} for more details.  The dual
isocrystal of \(\mathcal{L}_{\pi}\) is \(\mathcal{L}_{-\pi}\).
\end{definition}

\begin{situation}\label{situation}%
Let \(f_{1},\ldots,f_{r} \in \mathbb{F}_{q}[x_{1},\ldots,x_{n}]\).
Let
\(Z=\operatorname{Spec}(\mathbb{F}_{q}[x_{1},\ldots,x_{n}]/(f_{1},\ldots,f_{r}))\).
Introducing new variables \(x_{n+1},\ldots,x_{n+r}\), let
\(g = x_{n+1}f_{1} + \cdots + x_{n+r}f_{r}\).  We shall refer to the
rigid cohomology with twisted coefficient
\(\mathrm{H}^{\bullet}_{\mathrm{rig}}(\mathbb{A}_{\mathbb{F}_{q}}^{n+r};g^{\ast}\mathcal{L}_{\pi})\),
as the ``overconvergent Dwork cohomology'' of \(Z\).
\end{situation}

As we have said, the benefit to be able to work with the overconvergent
Dwork cohomology is that it has a rather explicit chain-level model.
Let us now elaborate on the de~Rham complex that is used to compute
the overconvergent Dwork cohomology.

\begin{construction}[Overconvergent Dwork complex]
\label{construction:exponential-model-complex}
Consider the Monsky--Washnitzer algebra
\begin{align*}
  B &= K\langle x_{1},\ldots,x_{n+r} \rangle^{\dagger}\\
  &=
    \left\{
    \sum_{u\in \mathbb{N}^{n+r}} c_{u} x^{u}  \in K[\![x_{1},\ldots,x_{n+r}]\!] : \exists \rho > 1,
    |c_{u}|\rho^{|u|} \xrightarrow{|u| \to \infty} 0
    \right\},
\end{align*}
where \(|u| = |u_{1}| + \cdots + |u_{n+r}|\), \(x^{u} = \prod_{i=1}^{n+r} x_{i}^{u_{i}}\).
For a subset \(I=\{i_{1},\ldots,i_{m}\}\) of \(\{1,\ldots,n+r\}\), with \(i_{1}<\cdots<i_{m}\),
we write \(x^{I} = \prod_{i\in I}x_{i}\),
and \(dx^{I} = dx_{i_{1}}\wedge\cdots\wedge dx_{i_{m}}\).
We use the notation \(\Omega^{m}\) to denote the space of ``overconvergent \(m\)-forms'',
that is,
\begin{equation}
\label{eq:decompose-differential-forms}
\Omega^{m} = \bigoplus_{\substack{I\subset\{1,\ldots,n+r\}\\|I|=m}} B\cdot dx^{I}
= \bigoplus_{\substack{I\subset\{1,\ldots,n+r\}\\|I|=m}} B_{I} \frac{dx^{I}}{x^{I}}
\end{equation}
where \(B_{I} = x^{I}B\).

Write \(g = \sum_{u\in \mathbb{N}^{n+r}} a_{u}x^{u}\), with
\(a_{u}\in \mathbb{F}_{q}\).  Let \(A_{u}\) be the Teichmüller lift of
\(a_{u}\) in \(W(\mathbb{F}_{q}) \subset \mathcal{O}_{K}\), and let
\(G = \sum_{u\in \mathbb{N}^{n+r}}A_{u}x^{u}\).  Then
\(G \equiv g \mod \pi\), and the overconvergent Dwork cohomology is
computed by the following exponentially twisted de~Rham complex:
\begin{equation}
\label{eq:dwork-complex}
\mathcal{D}^{\bullet}: \quad\Omega^{0} \xrightarrow{d + \pi dG} \Omega^{1} \xrightarrow{d + \pi dG} \cdots \xrightarrow{d + \pi dG} \Omega^{n+r}.
\end{equation}
That is, we have
\begin{equation*}
\mathrm{H}^{\ast}_{\mathrm{rig}}(\mathbb{A}^{n+r}_{\mathbb{F}_{q}},g^{\ast}\mathcal{L}_{\pi}) \simeq
\mathrm{H}^{\ast}(\mathcal{D}^{\bullet}).
\end{equation*}
\end{construction}

\begin{remark}%
The complex \(\mathcal{D}^{\bullet}\) is said to be exponentially
twisted because symbolically we have
\[
d + \pi dG = \exp(-\pi G) \circ d \circ \exp(\pi G).
\]
\end{remark}

\begin{remark}%
\label{remark:exponential-twist}
Dwork cohomology originated in Dwork's study of zeta functions of
projective hypersurfaces
\cite{dwork:zeta-function-hypersurface-1,dwork:zeta-function-hypersurface-2}.
The overconvergent Dwork cohomology
\(\mathrm{H}^{\ast}_{\mathrm{rig}}(\mathbb{A}^{n+r};g^{\ast}\mathcal{L}_{\pi})\)
is an overconvergent variant of Dwork's construction.  Dwork used
Banach spaces rather than weakly completed algebras in the sense of
Monsky--Washnitzer.  His construction was systematically generalized
in the context of toric exponential sums by Adolphson and Sperber
\cite{adolphson-sperber:exponential-sums-newton-polyhedra} (still
using Banach spaces instead of weakly completed versions).
Apparently, the nice properties of overconvergent Dwork cohomology was
first studied by Monsky~\cite{monsky:formal-cohomology-3}.
\end{remark}

\begin{construction}[Frobenius action on the overconvergent Dwork complex]
\label{construction:frobenius}
The inverse image isocrystal \(g^{\ast}\mathcal{L}_{\pi}\) has a
Frobenius structure, which can also be explained using the exponential
twist.  Let us denote by \(\sigma\) the endomorphism of \(B\)
defined by
\begin{equation*}
\sigma\colon \sum_{u\in \mathbb{N}^{n+r}} a_{u} x^{u} \mapsto \sum_{u\in\mathbb{N}^{n+r}} a_{u}^{\tau} x^{pu}.
\end{equation*}
Then the Frobenius structure on \(g^{\ast}\mathcal{L}_{\pi}\) with
respect to \(\sigma\) can be symbolically determined via exponential twist
(Remark~\ref{remark:exponential-twist}) as follows:
\begin{align*}
  \varphi(\xi)
  &= \prod_{u} \exp(-\pi A_{u} x^{u}) \cdot \sigma\left( \prod_{u} \exp(\pi A_{u} x^{u}) \cdot \xi \right) \\
  &= \prod_{u} \exp(\pi A_{u}^{\tau} x^{pu} -\pi A_{u} x^{u} ) \cdot \xi^{\sigma} \\
  &= \prod_{u} \theta(A_{u} x^{u})^{-1} \cdot \xi^{\sigma},
\end{align*}
where \(u\) ranges in \(\mathbb{N}^{n+r}\), and \(\theta\) is the
Dwork exponential defined above.  Since \(\theta\) is overconvergent,
it follows that \(\varphi\) indeed takes \(B\) into \(B\).  Recall \(\mathbb{F}_{q} = \mathbb{F}_{p^{a}}\).
Let
\begin{equation}
\label{eq:frobenius-function}
F_{1}(x) = \prod_{u} \theta(A_{u}x^{u}), \quad \text{and} \quad F_{a}(x) = \prod_{i=0}^{a-1}F^{\tau^{i}}_{1}(x^{p^{i}}).
\end{equation}
Then both \(F_{1}(x)\) and \(F_{a}(x)\) are overconvergent analytic
functions in \((x_{1},\ldots,x_{n+r})\).  Whence
\(\varphi = F_{1}^{-1} \circ \sigma\).

The Frobenius \(\varphi\) induces an operation on the spaces
\(\Omega^{m}\) of differential forms.
\begin{equation*}
\varphi^{(m)}\left(\sum_{\substack{I\subset\{1,\ldots,n+r\}\\|I|=m}}\xi_{I}(x) \frac{dx_I}{x_{I}}\right)
= p^{m}\varphi(\xi_{I}(x))\frac{dx_{I}}{x_{I}},\quad \xi_I \in B_{I}.
\end{equation*}
One checks that the above definition turns
\((\varphi^{(m)})_{m=0}^{n+r}\) into a chain map \(\varphi^{(\bullet)}\colon \mathcal{D}^{\bullet} \to \mathcal{D}^{\bullet}\):
\begin{equation*}
\begin{tikzcd}
\Omega^{0} \ar[rr,"d + \pi dG"] \ar[d,"\varphi^{(0)}"] & &  \Omega^{1} \ar[rr,"d + \pi dG"] \ar[d,"\varphi^{(1)}"]& &  \cdots \ar[rr,"d + \pi dG"] & &  \Omega^{n+r} \ar[d,"\varphi^{(n+r)}"] \\
\Omega^{0} \ar[rr,swap,"d + \pi dG"] & &  \Omega^{1} \ar[rr,swap,"d + \pi dG"] & &  \cdots \ar[rr,swap,"d + \pi dG"] & &  \Omega^{n+r}.
\end{tikzcd}
\end{equation*}
It induces the semilinear Frobenius map on rigid cohomology
\(\varphi\colon\mathrm{H}^{\ast}_{\mathrm{rig}}(\mathbb{A}^{n+r};g^{\ast}\mathcal{L}_{\pi})\to\mathrm{H}^{\ast}_{\mathrm{rig}}(\mathbb{A}^{n+r};g^{\ast}\mathcal{L}_{\pi})\).
\end{construction}

\begin{remark}
While the Frobenius operator \(\varphi\) behaves well after taking
cohomology, it is ill-suited for chain-level manipulations due to
being an ``expanding map'', not completely continuous in the sense of
Serre~\cite{serre:completely-continuous-endomorphisms-on-p-adic-banach-spaces},
that is, \(\varphi\) is not a uniform limit of finite rank linear
maps.

This analytic issue is common to F-isocrystals in general, not limited
to the Dwork isocrystal.  Consider the simplest example of the
standard Frobenius acting on the trivial isocrystal on
\(\mathbb{A}^{1}\), when \(q = p\).  To take advantage of
overconvergence, one is led to examine the Frobenius operation on
\(p\)-adic disks slightly larger than the unit disk.  However, the
\(p\)-power Frobenius map \(a \mapsto a^{p}\) takes the 1-dimensional
closed ball of radius \(|\pi|^{-1/N}\) to the \emph{larger} closed
ball of radius \(|\pi|^{-p/N}\).  The Frobenius pullback induces a map
on functions:
\[
\xi (x) \mapsto \xi(x^{p}) \colon \left\{ \sum a_{n} x^{n} : a_{n} \pi^{-n/N} \to 0 \right\}
\to \left\{ \sum b_{n} x^{n} : b_{n} \pi^{-pn/N} \to 0 \right\}.
\]
After extending the field \(K\) to a larger scalar field by adding
\(N\)\textsuperscript{th} roots of \(\pi\), an orthonormal basis of
the Banach space
\(\left\{\sum a_{n} x^{n} : a_{n} \pi^{-n/N} \to 0\right\}\) is given
by \(1, \pi^{1/N} x, (\pi^{1/N} x)^{2}, \ldots\), and an orthonormal
basis of \(\left\{\sum a_{n} x^{n} : a_{n} \pi^{-pn/N} \to 0\right\}\)
is given by \(1, \pi^{p/N} x, (\pi^{p/N} x)^{2}, \ldots\).  The
columns of the matrix representing the linear mapping
\(\xi(x) \mapsto \xi(x^{p})\) with respect to these bases have larger
and larger absolute values.  On the other hand, the norm of the column
vectors of a completely continuous operator should converge to \(0\).
\end{remark}

\begin{construction}[Dwork operators]\label{construction:dwork-operator}
To fix this, following Dwork, we
consider the following operator  \(\psi\) on power series:
\begin{equation}\label{eq:psi1}
\psi\left( \sum_{u\in \mathbb{N}^{n+r}} a_{u} x^{u} \right) = \sum_{u\in \mathbb{N}^{n+r}} a_{pu} x^{u}.
\end{equation}
Recall \(q=p^{a}\), and \(\tau^{a} = \mathrm{Id}\).  Then we have
\begin{align}\label{eq:psi2}
(\tau^{-1}\circ \psi)^{a}\left( \sum a_{u}x^{u} \right)
  &= \sum a_{qu}^{\tau^{-a}} x^{u} = \sum a_{qu}x^{u} \\
  &= \psi^{a}\left( \sum a_{u}x^{u} \right)\nonumber
\end{align}
as well as
\begin{equation*}
\eta \cdot  (\tau^{-1} \circ \psi)(\xi)
= (\tau^{-1}\circ\psi)(\eta^{\tau}(x^{p}) \cdot \xi(x)).
\end{equation*}

Form the composition
\begin{equation*}
\alpha_{1} = \tau^{-1} \circ \psi \circ F_{1}, \quad\text{as well as}\quad
\alpha_{a} = \alpha_{1}^{a}.
\end{equation*}
Then by \eqref{eq:psi1} and \eqref{eq:psi2},
\begin{align*}
  \alpha_{a} &= \alpha_{1}^{a}
  = \underbrace{(\tau^{-1} \circ \psi \circ F_{1}) \circ \cdots \circ(\tau^{-1} \circ \psi \circ F_{1})}_{a\text{ times}}, \\
  &= (\tau^{-1} \circ \psi)^{a}\circ \left( \prod_{i=0}^{a-1} F_{1}^{\tau^{i}}(x^{p^{i}}) \right) \\
  &= \psi^{a} \circ F_{a}.
\end{align*}

It is clear that \(\alpha_{1}\) preserves \(B_{I}\) for any \(I \subset \{1,2,\ldots,n+r\}\); and is a left inverse to
\(\varphi\): \(\alpha_{1}\circ\varphi = \mathrm{Id}\).  Similarly,
\(\alpha_{a}\) also preserves \(B_{I}\), and is a left inverse to the
\(a\textsuperscript{th}\) iteration of \(\varphi\).  Note that
\(\alpha_{1}\) is a \(\tau^{-1}\)-semilinear map on the infinite
dimensional \(K\)-vector space \(B_{I}\), and \(\alpha_{a}\) is a \(K\)-linear
operator on \(B_{I}\).  In the terminology of
Monsky~\cite[Definition~2.1]{monsky:formal-cohomology-3},
\(\alpha_{1}\) and \(\alpha_{a}\) are ``Dwork operators'' on the spaces
\(B_{I}\)
(the power series \(F_{1}(x)\) and \(F_{a}(x)\) are integrally defined, i.e., they are elements of
\(\mathcal{O}_{K}[\![x_{1},\ldots,x_{n+r}]\!] \cap B_{I}\)).  Moreover, both \(\alpha_{1}\) and
\(\alpha_{a}\) are ``nuclear operators'' on the space
\(B\) in the sense of Monsky~\cite[Theorem~2.1]{monsky:formal-cohomology-3}.  Therefore,
their ``characteristic power series'' \cite[Theorem~1.2]{monsky:formal-cohomology-3}
are well-defined.

The operators \(\alpha_{1}\) and \(\alpha_{a}\) extends to endomorphsms of
the Dwork complex.  We define
\(\alpha_{1}^{(m)} \colon \Omega^{m} \to \Omega^{m}\) by the following
formula:
\begin{equation*}
\alpha_{1}^{(m)}
\left(\sum_{\substack{I\subset\{1,\ldots,n+r\}\\|I|=m}}\xi_I(x) \frac{dx_I}{x_{I}}\right)
= \alpha_{1}(\xi_I(x)) \frac{dx_I}{x_I}, \quad \xi_I \in B_{I}.
\end{equation*}
Then \(\alpha_{1}^{(m)} \circ \varphi^{(m)} = p^{m} \mathrm{Id}\), and
the following diagram is commutative:
\begin{equation*}
\begin{tikzcd}
\Omega^{0} \ar[rr,"d + \pi dG"] & & \Omega^{1} \ar[rr,"d + \pi dG"] & & \cdots \ar[rr,"d + \pi dG"] & & \Omega^{n+r},\\
\Omega^{0} \ar[rr,swap,"d + \pi dG"] \ar[u,"\alpha_{1}^{(0)}"] & & \Omega^{1} \ar[rr,swap,"d + \pi dG"] \ar[u,swap,"p^{-1}\alpha_{1}^{(1)}"]& & \cdots \ar[rr,swap,"d + \pi dG"] & & \Omega^{n+r} \ar[u,"p^{-n-r}\alpha_{1}^{(n+r)}",swap].
\end{tikzcd}
\end{equation*}
Similarly, we may define
\(\alpha_{a}^{(m)}\colon\Omega^{m}\to\Omega^{m}\), satisfying
\(\alpha_{a}^{(m)}\circ(\varphi^{(m)})^{a}=q^{m}\mathrm{Id}\), and the
operators \(q^{-m}\alpha_{a}^{(m)}\) induce a chain map of
\(\mathcal{D}^{\bullet}\).  Thus, the operators
\((p^{-m}\alpha_{1}^{(m)})_{m=0}^{n+r}\) and
\((q^{-m}\alpha_{a}^{(m)})_{m=0}^{n+r}\) induce maps on overconvergent
Dwork cohomology:
\begin{equation*}
\alpha_{1}, \alpha_{a}\colon \mathrm{H}^{\ast}_{\mathrm{rig}}(\mathbb{A}^{n+r};g^{\ast}\mathcal{L}_{\pi})
\to \mathrm{H}^{\ast}_{\mathrm{rig}}(\mathbb{A}^{n+r};g^{\ast}\mathcal{L}_{\pi}).
\end{equation*}
Since \(\alpha_{a}^{(m)}\) (resp.~\(\alpha_{1}^{(m)}\)) is a left inverse
to \(\varphi^{(m)}\) (resp., \((\varphi^{(m)})^{a}\)) on the chain
level, and since the overconvergent Dwork cohomology is finite
dimensional, we find on the cohomology level, \(q^{-m}\alpha_{a}^{(m)}\) is
\emph{equal} to the inverse to \((\varphi^{(m)})^{a}\).  In particular
\begin{align}\label{eq:dwork-vs-frob}
&\det(1 - t\cdot q^{-m}\alpha_{a} | \mathrm{H}^{m}_{\mathrm{rig}}(\mathbb{A}^{n+r};g^{\ast}\mathcal{L}_{\pi})) \\
=& \det(1 - t\cdot(\varphi^{(m)})^{-a} | \mathrm{H}^{m}_{\mathrm{rig}}(\mathbb{A}^{n+r};g^{\ast}\mathcal{L}_{\pi})).\nonumber
\end{align}
\end{construction}

\begin{remark}[Fredholm determinant]%
For every subset \(I\) of \(\{1,2,\ldots,n+r\}\), the collection
\begin{equation}
\label{eq:basis}
\{x^u: u\in \mathbb{N}^{n+r}, u_i \geq 1 \text{ if } i \in I\}
\end{equation}
is a ``basis'' of \(B_I\) in the sense that every \(\xi \in B_I\) can
be written uniquely as an infinite linear combination
\(\xi = \sum a_u x^u\) for some \(a_u \in K\).  Furthermore, any
continuous linear operator on \(B_I\) can be represented by a unique
infinite matrix.

According to Monsky's theory~\cite[Theorem~1.6]{monsky:formal-cohomology-3},
the matrix associated to the Dwork operator \(\alpha_a\) has a
well-defined Fredholm determinant, which equals the
\emph{characteristic power series} he defined.  This is because that \(B_I\)
is a union of a certain \(\alpha_a\)-invariant Banach subspaces
\(B_I(b)\): \(B_I=\bigcup_{b>0} B_I(b)\).  Here, \(B_I(b)\) is the
space of rigid analytic functions converging on the closed polydisk of
radius \(|p|^{-b}\).  The operators \(\alpha_a|_{B_I(b)}\) are all
completely continuous in the sense of Serre~\cite{serre:completely-continuous-endomorphisms-on-p-adic-banach-spaces}.  An appropriately scaled version of
\eqref{eq:basis} is an orthonormal basis of the Banach space
\(B_I(b)\).  For this reason, in the remainder of this paper, we will
simply refer to the characteristic power series of \(\alpha_a|_{B_{I}}\) as
the \emph{Fredholm determinant} of \(\alpha_a\), denoted as
\(\det(1 - t\alpha_a| B_{I})\).

On occasion, we will also encounter other overconvergent spaces, such
as the spaces \(C\) and \(C/B\) (who will show up in the proof of
Theorem~\ref{theorem:dwrok-trace-formula}) among others.  In these
situations, when we refer to the term ``basis'', it should be
interpreted in a manner analogous to the one explained in the
preceding paragraph.  Consequently, the characteristic power series of Dwork
operators on these spaces can all be computed as the Fredholm
determinants of the infinite matrices associated to the specified
``bases''.  Ergo, we will not distinguish between the terms
``characteristic power series'' and ``Fredholm determinant''.
\end{remark}

Now that we have completed the groundwork concerning overconvergent
Dwork cohomology, we are ready to present a theorem of Baldassarri and
Berthelot that establishes a connection between overconvergent Dwork
cohomology and the rigid cohomology of \(Z\).  This theorem serves as
the cornerstone for our chain-level arguments.

\begin{theorem}[{\cite[Theorem~3.1]{baldassarri-berthelot:dwork-cohomology-for-singular-hypersurfaces}}]\label{theorem:comparison}
Notation as in Situation~\ref{situation}, we have a natural isomorphism
\begin{equation*}
\mathrm{H}^{\ast}_{\mathrm{rig}}(\mathbb{A}_{\mathbb{F}_{q}}^{n+r};g^{\ast}\mathcal{L}_{\pi})
\simeq \mathrm{H}^{\ast}_{\mathrm{rig},Z}(\mathbb{A}_{\mathbb{F}_{q}}^{n})
\end{equation*}
which is compatible with Frobenius actions.
\end{theorem}

In fact, the following local version of the theorem is true.

\begin{theorem}[{\cite[Theorem~2.14]{baldassarri-berthelot:dwork-cohomology-for-singular-hypersurfaces}}]%
\label{theorem:local-comparison}
Let notation be as in Situation~\ref{situation}.  Let
\(\varpi\colon \mathbb{A}^{n+r} \to \mathbb{A}^{n}\) be the projection
to the first \(n\) coordinates.  Let \(\mathcal{L}\) denote the
arithmetic \(\mathscr{D}\)-module associated to the Dwork crystal
\(g^{\ast}\mathcal{L}_{\pi}\) on \(\mathbb{A}^{n+r}\).  Then \(\varpi_{+}\mathcal{L}\) is
isomorphic to the local cohomology complex
\(\mathbb{R}\underline{\Gamma}_{Z}^{\dagger}(\mathcal{O}_{\widehat{\mathbb{P}}^{n},\mathbb{Q}}({}^{\dagger}H))[r]\).
\end{theorem}

Here,
\(\mathcal{O}_{\widehat{\mathbb{P}}^{n},\mathbb{Q}}({}^{\dagger}H)\)
is the sheaf of functions on the formal projective space
\(\widehat{\mathbb{P}}^{n}\) with overconvergent singularities along
the infinity hyperplane \(H \subset \mathbb{P}^{n}_{\mathbb{F}_{q}}\).
This is what was denoted as
\(\mathcal{O}_{\widehat{\mathbb{P}}^{n},\mathbb{Q}}(\infty)\) in
\cite{baldassarri-berthelot:dwork-cohomology-for-singular-hypersurfaces}.
See Example~\ref{example:structure-sheaf-of-an} for more details.

\begin{remark}
Theorems \ref{theorem:comparison} and \ref{theorem:local-comparison}
were originally proved by Baldassarri and Berthelot for the case where
\(Z\) is a complete intersection.  They noted
\cite[p.~208, Remark]{baldassarri-berthelot:dwork-cohomology-for-singular-hypersurfaces}
that this hypothesis was only necessary to verify an equality (2.14.4)
regarding the compatibility between the local cohomology functor and
the extraordinary inverse image functor.  This compatibility was later
established by Caro unconditionally
\cite[(2.2.18.1)]{caro:overcoherent-arithmetic-d-modules}.

Beware, however, that Caro's construction of the local cohomology
complex \cite[\S2]{caro:overcoherent-arithmetic-d-modules} differs
from Berthelot's approach when \(Z\) is not a divisor.  Nevertheless,
as noted in
\cite[Remarque~2.2.7]{caro:overcoherent-arithmetic-d-modules}, the two
constructions coincide when the local cohomology is taken with respect
to modules of the form
\(\mathcal{O}_{\mathcal{P},\mathbb{Q}}({}^{\dagger}H)\).  Also, the
proof of
\cite[Theorem~2.14]{baldassarri-berthelot:dwork-cohomology-for-singular-hypersurfaces}
remains valid using Caro's version of the local cohomology functor.

We were informed by Steve Sperber that Nobuo Tsuzuki also prepared a
proof of Theorem~\ref{theorem:comparison} in 2011 (unpublished). We
appreciate Sperber for sharing Tsuzuki's manuscript with us.
\end{remark}

\begin{remark}
If in the definition of Dwork cohomology one uses finite type rings
instead of using weakly completed algebra or Banach algebras, one gets
the so-called ``algebraic Dwork cohomology''.  The algebraic analogue
of Theorem~\ref{theorem:comparison} is well-known: it is proved by
N.~Katz~\cite{katz:differential-equations-satisfied-by-period-matrices}
when \(Z\) is a hypersurface; by
Adolphson--Sperber~\cite[]{adolphson-sperber:dwork-cohomology-affine-complete-intersection}
when \(Z\) is a \emph{smooth} complete intersections in a smooth
affine variety.  The algebraic analogue of Theorem~\ref{theorem:local-comparison} was shown by
Dimca~et~al.~\cite{dimca-maaref-sabbah-saito:dwork-cohomology-and-d-modules}
and
Baldassarri--D'Agnolo~\cite{baldassarri-d-agnolo:dwork-cohomology-algebraic-d-modules}.

Beside the work of Baldassarri and Berthelot, there are many previous
works aiming to provide a comparison between rigid cohomology spaces
and the cohomology spaces (and their variants) constructed by Dwork
and by Adolphson and Sperber, such as
Berthelot~\cite{berthelot:rigid-cohomology-and-dwork-theory-exponential-sums}
(for exponential sums over a 1-dimensional torus),
P.~Bourgeois~\cite[]{bourgeois:vanishing-and-purity-of-rigid-cohomology-assocaited-to-exponential-sums}
(for Newton nondegenerate toric exponentials sums), and Peigen
Li~\cite[]{li_peigen:exponential-sums-and-rigid-cohomology} (for toric
exponential sums).
\end{remark}

Theorem~\ref{theorem:comparison} implies a result on cohomology with compact supports
by taking Poincaré duality.  Before stating it let us recall some
standard notation about Tate twists.

For an integer \(w\), let \(K(w)\) be the \(1\)-dimensional vector
space over \(K\) equipped with a \(\tau\)-semilinear map
\(x\mapsto p^{-w}\tau(x)\).  For a finite dimensional \(K\)-vector
space \(M\) equipped with a bijective \(\tau\)-semilinear map
\(\varphi\), we use \(M(w)\) to denote the tensor product
\(M\otimes_{K}K(w)\), whose \(\tau\)-semilinear map is twisted by
\(p^{-w}\).  For \((M,\varphi)\) as above, we endow its dual space
\(M^{\ast}\) the transpose inverse of \(\varphi\).  With this
convention, we have \((M(w))^{\ast}\simeq M^{\ast}(-w)\) as spaces
equipped with \(\tau\)-semilinear maps.

\begin{corollary}%
\label{corollary:compact-support-frobenius}
Let \(f_{1},\ldots,f_{r} \in \mathbb{F}_{q}[x_{1},\ldots,x_{n}]\) be polynomials in \(n\) variables.
Let \(Z\)
be the closed subscheme of \(\mathbb{A}^{n}\) cut out by \(f_{1},\ldots,f_{r}\).
Let \(g = \sum_{i=1}^{r}x_{n+i}f_{i} \in \mathbb{F}_{q}[x_{1},\ldots,x_{n+r}]\).
Then we  have an isomorphism of rigid cohomology spaces compatible with
Frobenius actions:
\begin{equation*}
\mathrm{H}^{n-r+j}_{\mathrm{rig},c}(Z)(n) \xrightarrow{\sim} [\mathrm{H}^{n+r-j}_{\mathrm{rig}}(\mathbb{A}^{n+r};g^{\ast}\mathcal{L}_{\pi})]^{\ast}
\end{equation*}
In other words, under the isomorphism, the Frobenius operator
\(q^{-n}\mathrm{Frob}_{q}|_{\mathrm{H}^{n-r+j}_{\mathrm{rig},c}(Z)}\)
corresponds to the transpose of the inverse of the Frobenius operator
\((\varphi^{(n+r-j)})^{a}\) on the overconvergent Dwork cohomology.
Moreover, we have
\begin{equation*}
\det(1 - t\mathrm{Frob}_{q}|\mathrm{H}^{n-r+j}_{\mathrm{rig},c}(Z))
= \det(1 -  q^{j-r}t\alpha^{(n+r-j)}_{a}| \mathrm{H}_{\mathrm{rig}}^{n+r-j}(\mathbb{A}^{n+r};g^{\ast}\mathcal{L}_{\pi})).
\end{equation*}
\end{corollary}

\begin{proof}
Recall the statement of Poincaré duality for rigid cohomology (see
\cite[Corollary~8.3.14]{le-stum:rigid-cohomology}, and
\cite[Theorem~1.2.3]{kedlaya:finiteness-rigid-cohomology}; note the
latter article ignores the Tate twist): for a nonsingular,
geometrically connected variety \(X\) of dimension \(N\), and any
closed subvariety \(Y\) of \(X\), we have a perfect pairing
\begin{equation*}
\mathrm{H}^{i}_{\mathrm{rig},Y}(X) \otimes_K \mathrm{H}^{2N-i}_{\mathrm{rig},c}(Y) \to K(-N).
\end{equation*}

Applying Poincaré duality and Theorem~\ref{theorem:comparison}, to
\(X = \mathbb{A}^{n}\) and \(Y=Z\), we find the dual space of
\(\mathrm{H}^{n-r+j}_{\mathrm{rig},c}(Z)(n)\) is isomorphic to
\(\mathrm{H}^{n+r-j}_{\mathrm{rig},Z}(\mathbb{A}^{n})\simeq\mathrm{H}^{n+r-j}_{\mathrm{rig}}(\mathbb{A}^{n+r};g^{\ast}\mathcal{L}_{\pi})\),
in a way that commutes with semilinear Frobenius actions.  This
implies that the \(a\textsuperscript{th}\) iterations of these
semilinear maps, which are linear, also match.

The \(a\textsuperscript{th}\) iteration of the semilinear Frobenius
operator on \(\mathrm{H}^{n-r+j}_{\mathrm{rig},c}(Z)(n)\) is the
linear Frobenius operator twisted by \(q^{-n}\), i.e., the map
\(q^{-n}\mathrm{Frob}_{q}\colon\mathrm{H}^{n-r+j}_{\mathrm{rig},c}(Z)\to\mathrm{H}^{n-r+j}_{\mathrm{rig},c}(Z)\).
By the above discussion, it corresponds to the transpose of the
inverse of the \(a\textsuperscript{th}\) iteration of the semilinear
Frobenius operator \(\varphi^{(n+r-j)}\) on the overconvergent Dwork
cohomology
\(\mathrm{H}^{n+r-j}_{\mathrm{rig}}(\mathbb{A}^{n+r};g^{\ast}\mathcal{L}_{\pi})\).
Since the characteristic polynomial of the transpose of a linear
operator is identical to that of the operator itself, we conclude:
\begin{align*}
  \det(1-t\mathrm{Frob}_{q}| \mathrm{H}^{n-r+j}_{\mathrm{rig},c}(Z))
  &= \det(1 - tq^{n}(\varphi^{(n+r-j)})^{-a}| \mathrm{H}^{n+r-j}_{\mathrm{rig}}(\mathbb{A}^{n+r};g^{\ast}\mathcal{L}_{\pi})) \\
  \text{[By \eqref{eq:dwork-vs-frob}]} &= \det(1 - tq^{j-r}\alpha_{a}|\mathrm{H}^{n+r-j}_{\mathrm{rig}}(\mathbb{A}^{n+r};g^{\ast}\mathcal{L}_{\pi})).
\end{align*}
This completes the proof.
\end{proof}

\section{Visibility of Frobenius eigenvalues}
\label{sec:visibility}
Let \(f_{1},\ldots,f_{r} \in \mathbb{F}_{q}[x_1,\ldots,x_n]\) be a
collection of polynomials.  For every subset
\(I \subset \{1,2,\ldots,r\}\), set
\[
Z_{I}= \operatorname{Spec}\mathbb{F}_q[x_1,\ldots,x_n]/(f_i : i\in I),
\]
and \(Z_{I}^{\ast} = Z_{I} \cap \mathbb{G}_{\mathrm{m}}^{n}\).
Write \(Z=Z_{\{1,2,\ldots,r\}}\).

The purpose of this section is to prove Theorem~\ref{theorem:cancellation}, that
is, we want to show that the Frobenius eigenvalues of \(Z\) are weakly visible
in the zeta functions \(\zeta_{Z_I^{\ast}}(t)\).

The idea of the proof of Theorem~\ref{theorem:cancellation} can be
explained as follows:
\begin{enumerate}[wide]
\item By Corollary~\ref{corollary:compact-support-frobenius}, the overconvergent Dwork
cohomology associated to \(Z\) computes the rigid cohomology of \(Z\).
Hence the Frobenius eigenvalues of \(Z\) can be manifestly computed
using the operator \(\alpha_{a}\).  This is our starting point of the proof
of Theorem \ref{theorem:cancellation}.

\item The chain level operators \(\alpha_{a}|_{B_I}\) are
``nuclear operators''.
Since cohomology spaces are subquotients of \(\bigoplus B_{I}\),
Monsky's spectral theory (see Lemmas~\ref{lemma:nuclear-factor},
\ref{lemma:factor}) implies that the cohomological eigenvalues
are also ``eigenvalues'' of \(\bigoplus\alpha_{a}|_{B_I}\).
In the latter context, ``eigenvalue'' should be interpreted as the
reciprocal roots of the Fredholm determinant of \(\alpha_{a}\).  Thus,
Frobenius eigenvalues are visible in the Fredholm determinants
\(\det(1-t\alpha_{a}| B_I)\).
\item The spaces \(B_{I}\) are all \(\alpha_{a}\)-invariant subspaces of \(B\).
Thus \(\det(1-t\alpha_{a}| B)\) witnesses
all the ``chain level eigenvalues'' of \(\alpha_{a}|_{B_I}\).  The
Dwork trace formula, applying to \(\alpha_{a}|_B\), equates an alternating
product of \(\det(1-t\cdot q^{m}\alpha_{a}| B)\) with an alternating
product of zeta functions (see
Theorem~\ref{theorem:dwrok-trace-formula} and
Lemma~\ref{lemma:from-l-star-to-zeta}).  A M\"obius inversion (see the
formulae in Definition~\ref{definition:delta-operator}) then allows us
to represent \(\det(1-t\alpha_{a}| B)\) as an infinite product of zeta
functions. The Frobenius eigenvalues, which are visible in the
Fredholm determinant, are thereby weakly visible in the zeta
functions.
\end{enumerate}

Let us carry out the details.
To show the Fredholm determinants \(\det(1-t\alpha_{a} | B_I)\) contain all the
information of Frobenius eigenvalues, we need the following lemma,
extracted from \cite[Theorem~1.4]{monsky:formal-cohomology-3}.

\begin{lemma}%
\label{lemma:nuclear-factor}
Let \(X\) be a variable.
Let \(M^{\prime} \xrightarrow{f} M \xrightarrow{g} M^{\prime\prime}\) be a complex of nuclear \(K[X]\)-modules
in the sense of \textup{\cite[Definition~1.4]{monsky:formal-cohomology-3}}.
Then \(H=\operatorname{Ker}g/\operatorname{Im}f\) is also nuclear, and
\(\det(1 - t \cdot X | H)\) is
a factor of \(\det(1-t\cdot X | M)\).
\end{lemma}

\begin{proof}
We follow Monsky's notation in \cite[\S1]{monsky:formal-cohomology-3}.
By \cite[Theorem~1.4]{monsky:formal-cohomology-3}, \(H\) is a nuclear
\(K[X]\)-module.  It follows that for any bounded subset
(cf.~\cite[Definition 1.3]{monsky:formal-cohomology-3})
\(S \subset K[X]\setminus XK[X]\), we have a decomposition
\(H = N(S,H)\oplus F(S,H)\)
(\cite[Theorem~1.1]{monsky:formal-cohomology-3}).  By \cite[Proof of
Theorem~1.4, eighth line]{monsky:formal-cohomology-3}, \(N(S,H)\) is a
subquotient of \(N(S,M)\).  As both \(N(S,H)\) and \(N(S,M)\) are
finite dimensional, \(\det(1-tX | N(S,H))\) divides
\(\det(1-tX | N(S,M))\).  Taking limit by letting \(S\) run through
all bounded subsets, we conclude that \(\det(1-tX| H)\) is a factor
of \(\det(1-tX | M)\).
\end{proof}

\begin{lemma}%
\label{lemma:factor}
For each \(I \subset \{1,2,\ldots,n+r\}\), and each \(i \in \mathbb{Z}\),
\(\det(1-t\cdot q^{-i}\alpha_{a} | B_{I})\)
is a \(p\)-adic entire function. Moreover, the polynomial
\(\det(1-t\mathrm{Frob}_q| \mathrm{H}^{n-r+j}_{\mathrm{rig},c}(Z))\)
is a factor of
\[
\prod\limits_{|I|=n+r-j}\det(1-t\cdot q^{j-r}\alpha_{a} | B_{I}).
\]
\end{lemma}

\begin{proof}
All summands \(B_{I}\) of \eqref{eq:decompose-differential-forms} are
nuclear \(K[X]\)-modules, where the action of \(X\) coming from that
of \(\alpha_{a}\).  Since \(\alpha_{a}|_{B_{I}}\) is nuclear, the
series \(\det(1-t \cdot q^{-i}\alpha_{a} | B_I)\) is a \(p\)-adic
entire function by Monsky's theory.  Since
\(\Omega^{n+r-j} = \bigoplus_{|I|=n+r-j}B_{I} \frac{dx^{I}}{x^{I}}\),
\(\alpha^{(n+r-j)}_{a} = \bigoplus \alpha_{a}|_{B_{I}}\), and since
the overconvergent Dwork cohomology
\(\mathrm{H}^{n+r-j}(\mathbb{A}^{n+r};g^{\ast}\mathcal{L}_{\pi})\) is
computed as the cohomology of
\(\Omega^{n+r-j-1}\to\Omega^{n+r-j}\to\Omega^{n+r-j+1}\),
Corollary~\ref{corollary:compact-support-frobenius} and
Lemma~\ref{lemma:nuclear-factor} imply immediately that
\(\det(1-t\mathrm{Frob}_{q}|\mathrm{H}^{n-r+j}_{\mathrm{rig},c}(Z))\)
is a factor of
\(\prod\limits_{|I|=n+r-j}\det(1-t\cdot q^{j-r}\alpha_{a}|B_{I})\).
\end{proof}

\begin{corollary}%
\label{corollary:visibility}
Any Frobenius eigenvalue of \(\mathrm{H}^{\ast}_{\mathrm{rig},c}(Z)\) is weakly
visible in \(\det(1-t\alpha_{a}|B)\).
\end{corollary}

\begin{proof}
By Lemma~\ref{lemma:factor}, Frobenius eigenvalues of
\(\mathrm{H}^{n-r+j}_{\mathrm{rig},c}(Z)\) are weakly visible in the
product
\[\prod\limits_{|I|=j}\det(1-t\alpha_{a}|B_{I}).\]
The corollary follows since \(B_{I}\) is a nuclear submodule of \(B\).
\end{proof}

Next, let us explain how to read off the Fredholm determinants from the zeta
functions.

\begin{definition}\label{definition:delta-operator}%
  Following Dwork, we introduce an operation
  \(\delta\) on the set \(1 + t \mathbb{C}_{p}[\![t]\!]\)
  of formal power series with constant term one:
  \begin{equation*}
    \delta(\Gamma(t)) \overset{\mathrm{def}}{=\!=} \frac{\Gamma(t)}{\Gamma(qt)}.
  \end{equation*}
  The endomorphism \(\delta\) is invertible, and its inverse reads
  \begin{equation*}
    \delta^{-1}(\Gamma(t)) = \prod_{i=0}^{\infty} \Gamma(q^{i}t).
  \end{equation*}
\end{definition}

\begin{lemma}\label{lemma:entireness}
  If \(\Gamma(t) \in 1 + t \mathbb{C}_p[\![t]\!]\) is an entire function,
  then \(\delta^{-1}\Gamma(t)\) is also an entire function.
\end{lemma}

\begin{proof}
Let \(\Gamma(t) = \prod_{j=1}^{\infty}(1-\gamma_{i}t)\) be its infinite
product expansion, where the reciprocal zero $\gamma_j$ approaches zero as $j$ goes to infinity.
Then
\[\delta^{-1}(\Gamma(t)) = \prod_{i=0}^{\infty} \prod_{j=1}^{\infty} (1-q^i\gamma_jt)\]
is also such an infinity product whose reciprocal zero approaches zero.
\end{proof}

\begin{lemma}\label{lemma:moebius}
  Assume that \(\Gamma \in 1 +t \mathbb{C}_{p}[\![t]\!]\) is a \(p\)-adic meromorphic function on
  \(\mathbb{C}_{p}\), and \(\lambda \in \mathbb{C}_p\) is weakly visible in \(\Gamma\). Then
  \(\lambda\) is weakly visible in \(\delta(\Gamma)\).
\end{lemma}

\begin{proof}
Let \(Z(t) = \delta(\Gamma)\).  Then
\(\Gamma(t) = \prod_{i=0}^{\infty}Z(q^{i}t)\).  Write
\(Z(t) = u(t)/v(t)\), where \(u,v\) are entire functions, without
common zeros.  Then
\begin{equation*}
\Gamma(t) = \prod_{i=0}^{\infty} \frac{u(q^i t)}{v(q^i t)} = \frac{\delta^{-1}(u(t))}{\delta^{-1}(v(t))}.
\end{equation*}
By Lemma~\ref{lemma:entireness}, \(\delta^{-1}u(t)\) and
\(\delta^{-1}v(t)\) are entire.  If \(q^{m}\lambda\) is a reciprocal
zero of \(\Gamma(t)\), then \(1-q^{m}\lambda t\) must be a factor of
the infinite product \(\delta^{-1}(u(t))=\prod_{i}u(q^{i}t)\).  Hence
\(q^{m}\lambda\) is a reciprocal zero of \(u(q^{i}t)\) for some \(i\),
i.e., \(q^{m-i}\lambda\) is a reciprocal zero of \(Z(t)\).  Thus
\(\lambda\) is weakly visible in \(Z(t)\).  The polar case is similar.
\end{proof}

At this point, we recall the Dwork trace formula \cite[Lemma~2,
p.~637]{dwork:rationality-zeta-function}.
Remember that for each choice of \(\pi\),
the value of the overconvergent function \(\theta(t) = \exp(\pi t - \pi t^{p})\)
at \(t = 1\) is a primitive \(p\)\textsuperscript{th} root of unity, and
the function
\begin{equation}\label{eq:additive-character}
\Psi(t) = \theta(1)^{\mathrm{Tr}_{\mathbb{F}_{q}/\mathbb{F}_{p}}(t)}
\end{equation}
is a nontrivial additive character on \(\mathbb{F}_{q}\).

\begin{theorem}[Dwork trace formula]%
\label{theorem:dwrok-trace-formula}
For each positive integer $m$, define the $m$\textsuperscript{th} toric exponential sum
\[
S^{\ast}_m(g)=\sum\limits_{x\in\mathbb{G}_{\mathrm{m}}^{n+r}(\mathbb{F}_{q^m})}(\Psi\circ\operatorname{Tr}_{\mathbb{F}_{q^m}/\mathbb{F}_q})(g(x)).
\]
Then
\begin{equation*}
L^{*}(t)\overset{\mathrm{def}}{=\!=}\exp\left\{ \sum_{m=1}^{\infty} S_{m}^{\ast}(g)\frac{t^{m}}{m} \right\}
= \left\{ \delta^{n+r}\big(\det(1 - t\alpha_{a}| B)\big) \right\}^{(-1)^{n+r-1}}.
\end{equation*}
\end{theorem}

\begin{proof}
Let \(C = K\langle x^{\pm 1}_{1},\ldots,x_{n+r}^{\pm 1} \rangle^{\dagger}\)
be the weak completion of \(\mathcal{O}_{K}[x_{1}^{\pm 1},\ldots,x_{n+r}^{\pm 1}]\) with \(\pi\) inverted.
Then the rigid cohomology \(\mathrm{H}^{\ast}(\mathbb{G}_{\mathrm{m}}^{n+r};g^{\ast}\mathcal{L}_{\pi})\)
is computed by the exponentially twisted de~Rham complex
\begin{equation*}
C \to \bigoplus_{|I| = 1} C \frac{dx^{I}}{x^{I}} \to \cdots \to \bigoplus_{|I|=m} C \frac{dx^{I}}{x^{I}} \to \cdots,
\end{equation*}
whose differentials are given by \(d+\pi dG\) as in
\eqref{eq:dwork-complex}, and it is equipped with Dwork operators
\(q^{-m}\alpha_{a}^{(m)}\) as in Construction~\ref{construction:dwork-operator}.  By
\cite[Théorème~6.3II]{etesse-le-stum:l-functions-associated-to-overconvergent-f-isocrystals-1}
(noticing that \(\theta^{i}\) there is our \(q^{-i}\alpha_{a}^{(i)}\)),
the L-function for the isocrystal
\(g^{\ast}\mathcal{L}_{\pi}\)---which is \(L^{\ast}(t)\) in our
case---equals
\[
\prod_{i=0}^{n+r}\det(1 - t q^{n+r-i}\alpha_{a}^{(i)} |
\mathrm{H}^{i}_{\mathrm{rig}}(\mathbb{G}_{\mathrm{m}}^{n+r};g^{\ast}\mathcal{L}_{\pi}))^{(-1)^{i+1}}
\]
(note that \(\mathrm{H}^{i}_{\mathrm{rig}}(\mathbb{G}_{\mathrm{m}}^{n+r},g^{\ast}\mathcal{L}_{\pi})=0\) for \(i > n+r\)).
Since \(\alpha_{a}^{(m)}\) are nuclear operators,
the above quantity can also be written as
\begin{equation}
\label{eq:trace-formula}
\prod_{i=0}^{n+r}\det(1 - tq^{n+r-i}\alpha_{a}| C^{\binom{n+r}{i}})^{(-1)^{i+1}} =
\left\{\delta^{n+r}\left( \det( 1 - t\alpha_{a} | C) \right)\right\}^{(-1)^{n+r-1}}.
\end{equation}

If we can substitute \(C\) with \(B\) in \eqref{eq:trace-formula}, we
can readily derive the desired result.  It suffices to show that
\(\det(1 - t\alpha_{a} | C/B) = 1\).  Recall the formula
\(\det(1 - t \gamma) = \exp\left( -\sum_{m=1}^{\infty}
\frac{\mathrm{Tr}(\gamma^m){m}}t^m \right)\).  We only have to show
\(\mathrm{Tr}(\alpha_{a}^{m}| C/B)=0\) for all positive integers
\(m\).  A ``basis'' for \(C/B\) is comprised by the images of
\(x^{u}\) in \(C/B\), where
\(u \in \mathbb{Z}^{n+r} \setminus \mathbb{N}^{n+r}\).  In particular,
these \(u\) are never zero.  Recall that
\(\alpha_{a} = \psi^{a} \circ F_{a}\) (See
\eqref{eq:frobenius-function},
Construction~\ref{construction:dwork-operator}).  Because
\(F_{a} \in B\), the monomials in its power series expansion are
\(x^{v}\) with \(v \in \mathbb{N}^{n+r}\).  Thus we have
\begin{equation*}
\alpha_{a}^{m}(x^{u}) = \sum_{v} b_{v} x^{(u+v)/q^{m}},
\end{equation*}
in which \(b_{v}=0\) if \(q^{m}\nmid u+v\), or if
\(v \notin \mathbb{N}^{n+r}\).  To compute the trace of
\(\alpha_{a}^{m}|_{C/B}\), we need to examine the coefficient of \(x^{u}\)
in \(\alpha_{a}^{m}(x^{u})\) for \(u \notin \mathbb{N}^{n+r}\).  If
\(u = (u+v)/q^{m}\), then we have \(v = (q^{m}-1)u\) in
\(\mathbb{Z}^{n+r}\).  Since \(u \notin \mathbb{N}^{n+r}\), we have
\(v \notin \mathbb{N}^{n+r}\).  Therefore, if \(u = (u+v)/q^{m}\), we
must have \(b_{v}=0\).  From this, we conclude that the diagonal
entries of the matrix representation of the operator
\(\alpha_{a}^{m}|_{C/B}\) on the quotient with respect to the ``basis''
\begin{equation*}
\{\text{The image of }x^{u}\text{ in }C/B : {u\in \mathbb{Z}^{n+r}\setminus\mathbb{N}^{n+r}}\}
\end{equation*}
are all zero.  The theorem follows.
\end{proof}

Combining Corollary~\ref{corollary:visibility}, Lemma~\ref{lemma:moebius}, and
Theorem~\ref{theorem:dwrok-trace-formula}, we infer that any Frobenius
eigenvalue of \(\mathrm{H}^{\bullet}_{\mathrm{rig},c}(Z)\) is weakly visible in
\(L^{\ast}(t)\).

Our next objective is to relate the function \(L^{\ast}(t)\)
to the zeta functions of \(Z_{I}\).




\begin{lemma}%
  \label{lemma:from-l-star-to-zeta}
  We have
  \begin{equation*}
    L^*(t) = \prod_{J\subset \{1,2,\ldots,r\}}
    \zeta_{Z^*_{J}}(q^{|J|}t)^{(-1)^{r-|J|}}.
  \end{equation*}
\end{lemma}

\begin{proof}
  For any subset \(J\) of \(\{1,2,\ldots,r\}\),
  let \(\mathbb{A}^{J}\) denote the affine space with coordinates
  \((y_{j})_{j\in J}\), and \(g_{J}\) the regular function
  \(\sum_{j\in J}y_{j}f_{j}(x)\).
  The orthogonality of characters implies
  \[
    S_{J}
    \overset{\text{def}}{=\!=}
    \sum\limits_{x\in \mathbb{G}_{\mathrm{m}}^{n}(\mathbb{F}_q)}\sum_{y\in\mathbb{A}^{J}(\mathbb{F}_q)}
    \Psi(g_{J}(x,y)) =q^{|J|}|Z^*_{J}(\mathbb{F}_q)| \]
Since \(\mathbb{A}^{n} = \mathbb{G}_{\mathrm{m}}^{n} \sqcup D\), where \(D\) is
  the union of all coordinate hyperplanes, and since \(D\) admits the following
  standard semisimplicial resolution
  \begin{equation*}
    \begin{tikzcd}
      \cdots \ar[r,shift left=2] \ar[r] \ar[r,shift right=2]&
      \bigsqcup\limits_{\substack{J\subset \{1,2,\ldots,r\} \\ |J| = r-2}} \mathbb{A}^{J} \ar[r,shift left] \ar[r,shift right] & \bigsqcup\limits_{\substack{J\subset \{1,2,\ldots,r\} \\ |J| = r-1}}\mathbb{A}^{J} \ar[r] &D,
    \end{tikzcd}
  \end{equation*}
  we deduce from the inclusion-exclusion that
  \begin{align*}
    \sum_{x\in \mathbb{G}_{\mathrm{m}}^{n+r}(\mathbb{F}_q)}\Psi(g(x))
    &= \sum_{J\subset \{1,2,\ldots,r\}}(-1)^{r-|J|} S_J\\
    &= \sum_{J\subset \{1,2,\ldots,r\}}
      (-1)^{r-|J|}q^{|J|} |Z_{J}^*(\mathbb{F}_q)|.
  \end{align*}
  Replacing \(\mathbb{F}_q\) by \(\mathbb{F}_{q^m}\) and  \(\Psi\) by \(\Psi\circ \mathrm{Tr}_{\mathbb{F}_{q^m}/\mathbb{F}_q}\) in the above calculation,
  and exponentiating,
  we get the desired result.
\end{proof}

\begin{proof}[Proof of Theorem~\ref{theorem:cancellation}]
  By Corollary~\ref{corollary:visibility},
  the Frobenius eigenvalues of \(\mathrm{H}^{\ast}_{\mathrm{rig},c}(Z)\)
  are weakly visible in \(\det(1-t\alpha_{a}| B)\).
  By Lemma~\ref{lemma:moebius} and Theorem~\ref{theorem:dwrok-trace-formula},
  the Frobenius eigenvalues are weakly visible in \(L^{\ast}(t)\).
  By Lemma~\ref{lemma:from-l-star-to-zeta},
  they are weakly visible in
  \(\{\zeta_{Z_J^*}(t) : J \subset \{1,2,\ldots,r\}\}\).
\end{proof}

\begin{remark}
  The same proof also works for subvarieties in \(\mathbb{G}_{\mathrm{m}}^{n}\).
  That is, if \(\{f_1,\ldots,f_r\}\) is a set of Laurent polynomials, and
  \(Z_{I}\) is the vanishing scheme of \((f_i)_{i\in I}\),
  then the Frobenius eigenvalues of \(Z_{\{1,2,\ldots,r\}}\) is weakly visible
  in the set \(\{\zeta_{Z_I}(t): I\subset \{1,2,\ldots,r\}\}\).
\end{remark}

\section{Reciprocal roots of Fredholm determinants}
\label{sec:artin-hasse}

To establish Theorem~\ref{theorem:beyond-middle-dimension} and
Theorem~\ref{theorem:before-middle-dimension}, we will provide a more
robust result at the chain level.  In this section, we will apply
Adolphson--Sperber's theory to the specific context of our interest
and deduce a lower bound of the \(q\)-orders of ``eigenvalues'' of the operator
\(\alpha_{a}\colon B_{I} \to B_{I}\) introduced in
Construction~\ref{construction:dwork-operator}.  Given that rigid
cohomology are subquotients of the entries in the overconvergent Dwork
complex~\eqref{eq:dwork-complex}, divisibility at the chain level will
imply divisibility at the cohomological level, as per
Lemma~\ref{lemma:factor}.

\begin{notation}\label{notation}
Assume that we are provided with a sequence of polynomials
\(f_{1},\ldots,f_{r} \in \mathbb{F}_{q}[x_{1},\ldots,x_{n}]\), where
the degree of \(f_{i}\) is denoted as \(d_{i}\).  These polynomials
are ordered such that \(d_{1}\geq d_{2} \geq \cdots \geq d_{r}\).  Let
\(Z\) be the vanishing scheme of these polynomials
\(f_{1},\ldots,f_{r}\) in the affine space
\(\mathbb{A}^{n}_{\mathbb{F}_{q}}\).

For a given subset \(I\) of \(\{1,2,\ldots,n+r\}\), we define
\(I^{\prime}\) as \(I \cap\{1,2,\ldots,n\}\) and \(I^{\prime\prime}\)
as \(I \cap \{n+1,\ldots,n+r\}\).

Recall that \(B_{I}\) represents the subspace \(x^{I}B\) within the
Monsky-Washnitzer algebra \(B\) (refer to
Construction~\ref{construction:exponential-model-complex}),
and \(\alpha_{1}\), \(\alpha_{a}\) are the Dwork operators acting on \(B_{I}\)
(Construction~\ref{construction:dwork-operator}).
\end{notation}

The main result of this section is the following lemma, which is a
straightforward consequence of Adolphson--Sperber's result
(Theorem~\ref{theorem:adolphson-sperber}) once we explained the
relations between two different representations of Frobenius
structures.

\begin{lemma}%
\label{lemma:fredholm-determinant-of-alpha-on-bi}
Let notation be as in Construction~\ref{construction:dwork-operator}.
Let \(\lambda\) be a reciprocal root of the \(p\)-adic entire function
\(\det(1-t\alpha_{a}| B_I)\).
Then
\begin{equation*}
\operatorname{ord}_q \lambda \geq \frac{1}{d_1}\left( |I^{\prime}| + \sum_{i\in I^{\prime\prime}} (d_1-d_i)\right)
= \frac{1}{d_1}\left( |I| + \sum_{i\in I^{\prime\prime}} (d_1-d_i-1) \right),
\end{equation*}
where \(\operatorname{ord}_q\) is the \(p\)-adic valuation normalized so that
\(\operatorname{ord}_q(q)=1\).
\end{lemma}

\begin{remark}%
The combination of
Lemma~\ref{lemma:fredholm-determinant-of-alpha-on-bi} and
Lemma~\ref{lemma:factor} already provides interesting bounds on the
\(q\)-orders of the Frobenius eigenvalues of \(Z\) in all
cohomological degrees. However, upon closer examination of these
bounds, one finds that they are not as strong as asserted by our main
theorems.  To achieve stronger results, we will need to further cut
down some excess contributions, with the assistance of an algebraic
lemma, Lemma~\ref{lemma:new-sequence}.  Additionally, we will require
some arguments to elevate the \(q\)-order estimates to
\(q\)-divisibility bounds in the ring of algebraic integers.  These
steps will be addressed in the later sections.
\end{remark}

\begin{remark}%
Lemma~\ref{lemma:fredholm-determinant-of-alpha-on-bi} relies on the
work of Adolphson and
Sperber~\cite[Proposition~4.2]{adolphson-sperber:chevalley-warning},
which uses a more complicated representation of the Dwork operator.
In this remark, we will provide an informal explanation of why such a
more intricate choice is necessary.  For simplicity, we will focus on
the case when \(q=p\), making \(\alpha_{a} = \alpha_{1}\).

Our objective is to provide sharp lower bounds for the \(q\)-orders,
or equivalently, upper bounds for the \(p\)-adic absolute values, of
the reciprocal roots of the Fredholm determinant
\(\det(1-t\alpha_{1} | B_{I}) = \sum_{m} t^{m} \cdot
\mathrm{Tr}(\wedge^{m}\alpha_{1})\).  By the theory of Newton
polygons, we need to give upper bounds for
\(|\mathrm{Tr}(\wedge^{m}\alpha_{1})|\).

We can view \(\alpha_{1}\) as an infinite matrix by fixing the
standard ``basis'' of the infinite-dimensional \(K\)-vector space
\(B_{I}\) comprised of the monomials:
\begin{equation}
\label{eq:basis-again}
\{x_{1}^{u_{1}}\cdots x_{n+r}^{u_{n+r}}:u_{j}\geq0\text{ and }u_{i}\geq 1\text{ if }i\in I\}.
\end{equation}
A straightforward computation shows that the infinite matrix
representing \(\alpha_{1}\) is \([F_{pu-v}]\)
(\(u_{j},v_{j}\geq0, u_{i},v_{i}\geq 1\) if \(i \in I\)), where
\(F_{1}(x) = \sum_{u\in\mathbb{N}^{n+r}} F_{u} x^{u}\) is the power
series expansion of the function \(F_{1}\)
\eqref{eq:frobenius-function} that defines the Frobenius structure.
The number \(\mathrm{Tr}(\wedge^{m}\alpha_{1})\) is just the sum of
\(m \times m\) principal minors of the infinite matrix
\([F_{pu-v}]\). Therefore, ultimately we need a good upper bound for
\(|F_{u}|\).

Since the function \(F_{1}\) is derived from the Dwork exponential
\(\theta(t) = \exp(\pi t - \pi t^{p})\), to get an upper bound for
\(|F_{u}|\), we need an upper bound for the coefficient \(a_{j}\) of
\(t^{j}\) in the power series expansion of \(\theta\).  The most
optimal bound is
\begin{equation}
\label{eq:dwork-exp-bound}
|a_{j}| \leq |p|^{\frac{j(p-1)}{p^{2}}}
\end{equation}
(cf.~\cite[Eq.~(4.7) and p.~57, line 1]{dwork:zeta-function-hypersurface-1}).
But this estimate falls short of meeting our need (see
\eqref{eq:as-bound} below).


Therefore, the matrix of the Dwork operator with respect to the naive
``basis'' \eqref{eq:basis-again} does not yield an ideal estimate.
From an analytic perspective, the key to improving the estimate lies
in finding a suitable ``basis'' for \(B\), such that the matrix of the
Dwork operator under the new ``basis'' has smaller entries.
Equivalently, we would like to construct an isomorphism
\(\rho\colon B \to B\), inducing the following commutative diagram:
\[
\begin{tikzcd}
B \ar{r}{\rho} \ar{d}{\alpha_{1}} & B \ar[dashed]{d}{\beta_{1}} \\
B \ar{r}{\rho} & B
\end{tikzcd};
\]
and we hope the matrix of \(\beta_{1}\) under the ``basis''
\eqref{eq:basis-again} should have smaller entries than those of
\(\alpha_{1}\).  Additionally, we expect \(\beta_{1}\) should map
\(B_{I}\) to \(B_{I}\).

\medskip%
For every choice of \(s=1,2,\ldots,\infty\), Dwork was able to
construct a ``splitting function'' of level
\(s\)~\cite[Eq.~(4.1)]{dwork:zeta-function-hypersurface-1}.  These
splitting functions then give rise to Dwork operators \(\alpha_{1,s}\)
on \(B\), \(B_{I}\), and hence \(\Omega^{m}\).  All of them can play
the role of \(\beta_{1}\) above.  For \(s=1\), we are reduced to the
exponentially twisted \(\alpha_{1}\) we have been using so far.  It is
the choice \(s = \infty\)---corresponding to what we refer to as the
\emph{Artin--Hasse representation} of the Dwork crystal---that has the
most optimal convergence property (cf.~\eqref{eq:inf-level-bound}
below) among all splitting functions, and was extensively employed in
the work of Adolphson and Sperber.  Below we will introduce Dwork's
construction for \(s=\infty\).
\end{remark}

\begin{construction}
Let \(E(t) = \exp\left\{ \sum_{j=0}^{\infty} t^{p^{j}}/p^{j} \right\}\)
be the Artin--Hasse exponential.
It is well-known that \(E(t) \in \mathbb{Z}_{(p)}[\![t]\!]\).
By the theory of Newton polygon, the series \(\sum_{j=0}^{\infty}t^{p^{j}}/p^{j}\)
has a root \(\gamma\) satisfying \(|\gamma| = |\pi| = |p|^{\frac{1}{p-1}}\).
We define \(\theta_{\infty}(t) = E(\gamma\cdot t)\).  Write
\begin{equation*}
\theta_{\infty}(t) = \sum_{j=0}^{\infty} c_{j}t^{j},
\end{equation*}
then we have (cf.~\eqref{eq:dwork-exp-bound})
\begin{equation}\label{eq:inf-level-bound}
|c_{j}| \leq |\pi|^{j} = |p|^{\frac{j}{p-1}}.
\end{equation}
The twisting function \(\exp(\pi t)\) appeared in the overconvergent
Dwork complex may be explained as the infinite product
\begin{equation*}
\widehat{\theta}(t) = \exp(\pi t) = \prod_{j=0}^{\infty} \theta(t^{p^{j}}),
\end{equation*}
Dwork introduced the twisting factor associated with
\(\theta_{\infty}\) as
\begin{equation*}
\widehat{\theta}_{\infty}(t) = \prod_{j=0}^{\infty} \theta_{\infty}(t^{p^{j}})
= \exp\left\{ \sum_{j=0}^{\infty} \gamma_{j} t^{p^{j}} \right\},\quad
\text{so we have }\theta_{\infty}(t) = \widehat{\theta}_{\infty}(t)/\widehat{\theta}_{\infty}(t^{p}).
\end{equation*}
Here, \(\gamma_{j}=\sum_{i=0}^{j}\gamma^{p^{i}}/p^{i}\).  It is easy
to show \(\widehat{\theta}_{\infty}(t)\) is a rigid analytic function
on the open unit disk bounded by \(1\).
See~\cite[Eq.~(4.13)]{dwork:zeta-function-hypersurface-1}.  For a
conceptual explanation of this unusual exponential function in
connection with Frobenii liftings on the formal multiplicative group,
we recommend reading Pulita's
article~\cite{pulita:rank-one-solvable-p-adic-differential-equations}.
\end{construction}

\begin{construction}[Dwork operators \(\beta_{1}\) and \(\beta_{a}\) associated to the Artin--Hasse exponential]
\label{construction:better-dwork-operator}
Recall the meaning of
\(G(x) = \sum_{u\in\mathbb{N}^{n+r}} A_{u} x^{u}\), \(B\), \(B_{I}\)
given in Construction~\ref{construction:exponential-model-complex},
and the definitions of the operators \(\psi\), \(\alpha_{1}\), \(\alpha_{a}\)
given in Construction~\ref{construction:dwork-operator}.  Also recall
that \(\alpha_{1}\) was defined as an exponential twist, i.e., the
following diagram is commutative
\begin{equation*}
\begin{tikzcd}
B_{I} \ar[d,hook,"\cdot\widehat{\theta}",swap] \ar[r,"\alpha_{1}"] & B_{I} \ar[d,hook,"\cdot\widehat{\theta}"] \\
K[\![x_{1},\ldots,x_{n+r}]\!] \ar[r,"\tau^{-1}\circ\psi"] & K[\![x_{1},\ldots,x_{n+r}]\!]
\end{tikzcd},
\end{equation*}
where the vertical arrows send \(\xi\) to
\(\xi\cdot\prod_{u}\widehat{\theta}(A_{u}x^{u})\).

Define the Dwork operator associated to the ``Artin--Hasse
representation'' as a different twist using the commutativity of the
following diagram
\begin{equation*}
\begin{tikzcd}
B_{I} \ar["\cdot\widehat{\theta}_{\infty}",swap,d,hook] \ar[r,"\beta_{1}"] &
B_{I} \ar[d,"\cdot\widehat{\theta}_{\infty}",hook] \\
K[\![x_{1},\ldots,x_{n+r}]\!] \ar[r,"\tau^{-1}\circ\psi"] & K[\![x_{1},\ldots,x_{n+r}]\!]
\end{tikzcd},
\end{equation*}
where the vertical arrows are now given by
\(\xi \mapsto \xi \cdot \prod_{u}\widehat{\theta}_{\infty}(A_{u}x^{u})\).
Using the relation between \(\theta_{\infty}\) and \(\widehat{\theta}_{\infty}\),
we have
\begin{equation*}
\beta_{1}(\xi) =  \left(\psi(\xi\cdot \Phi)\right)^{\tau^{-1}} ,
\end{equation*}
where \(\Phi(x) = \prod_{u} \theta_{\infty}(A_{u}x^{u})\).  Since
\(B_{I}\) are ideals of \(B\), and \(\Phi\) is overconvergent by
definition, it is clear that \(\beta_{1}\colon B_{I} \to B_{I}\) is
well-defined for any subset \(I\) of \(\{1,2,\ldots,n+r\}\).  The
operator \(\beta_{1}\) is a Dwork operator in the sense of
Monsky~\cite{monsky:formal-cohomology-3} since
\(\Phi \in \mathcal{O}_{K}[\![x_{1},\ldots,x_{n+r}]\!] \cap B\).  We
define the operator \(\beta_{a}\) as the \(a\textsuperscript{th}\)
iteration of \(\beta_{1}\): \(\beta_{a} = \beta_{1}^{a}\).  Then
\(\beta_{1} \colon B_{I} \to B_{I}\) is \(\tau^{-1}\)-semilinear, and
\(\beta_{a}\) is \(K\)-linear.
\end{construction}

The following lemma tells us that the we can use the operator
\(\beta_{a}\) to study the Frobenius eigenvalues of the rigid cohomology.

\begin{lemma}%
\label{lemma:same-fredholm-determinant}
Notation as in Construction~\ref{construction:better-dwork-operator},
we have \(\det(1 - t\alpha_{a}| B_{I}) = \det(1 - t\beta_{a} | B_{I})\).
\end{lemma}

\begin{proof}
Since \(\alpha_{1}\) and \(\beta_{1}\) are both defined as twists of
\(\tau^{-1}\circ\psi\).  We have a commutative diagram
\begin{equation*}
\begin{tikzcd}
B_I \ar[r,hook] \ar[d,"\alpha_{1}"]& K[\![x_1,\ldots,x_{n+r}]\!] \ar[r,"\rho"] \ar[d,"\widetilde{\alpha}_1"]
& K[\![x_1,\ldots,x_{n+r}]\!]  \ar[d,"\widetilde{\beta}_1"] & B_I \ar[l,hook]  \ar[d,"\beta_1"] \\
B_I \ar[r,hook] & K[\![x_1,\ldots,x_{n+r}]\!] \ar[r,"\rho"] & K[\![x_1,\ldots,x_{n+r}]\!] & B_{I} \ar[l,hook]
\end{tikzcd},
\end{equation*}
where the hook arrows are the natural inclusions, and,
for a power series \(\xi \in K[\![x_1,\ldots,x_{n+r}]\!]\),
\begin{align*}
  \widetilde{\alpha}_1(\xi) &= \left\{ \prod_{u}\widehat{\theta}(A_u x^u) \right\}^{-1}\cdot \tau^{-1}\left(\psi\left(\prod_{u}\widehat{\theta}(A_u x^u)\cdot \xi\right)\right),\\
  \widetilde{\beta}_1(\xi) &= \left\{ \prod_{u}\widehat{\theta}_{\infty}(A_u x^u) \right\}^{-1}\cdot \tau^{-1}\left(\psi\left(\prod_{u}\widehat{\theta}_{\infty}(A_u x^u)\cdot \xi\right)\right),\\
  \rho(\xi) &= \xi \cdot \frac{\prod_u \widehat{\theta}(A_u x^u)}{\prod_u \widehat{\theta}_{\infty}(A_u x^u)}.
\end{align*}
It suffices to show the ratio of twisting factors
\(\prod_{u}\widehat{\theta}(A_{u}x^{u})/\widehat{\theta}_{\infty}(A_{u}x^{u})\),
and its reciprocal, are overconvergent, i.e., are elements of \(B\).  This will imply that \(\rho\) and \(\rho^{-1}\)
take \(B_I\) into \(B_I\), as \(B_I\) is an ideal of \(B\).
By transporting the structures, the
\(a\textsuperscript{th}\) iterations \(\alpha_a\) and \(\beta_{a}\) of
\(\alpha_1\) and \(\beta_1\) will correspond under
\(\rho\colon B_I \xrightarrow{\sim} B_I\), and will have the same
Fredholm determinant.

To prove the overconvergence of
\(\prod_{u}\widehat{\theta}(A_{u}x^{u})/\widehat{\theta}_{\infty}(A_{u}x^{u})\),
it in turn suffices to prove
\begin{equation*}
\widehat{\theta}(t)/\widehat{\theta}_{\infty}(t)\quad\text{and} \quad
\widehat{\theta}_{\infty}(t)/\widehat{\theta}(t)
\end{equation*}
are overconvergent, i.e., are elements of
\(K\langle t \rangle^{\dagger}\).  For a documented proof of this
basic fact, we refer the reader to Peigen Li's article
\cite[Proposition~2.1(i)]{li_peigen:exponential-sums-and-rigid-cohomology},
where Li also works out the radius of convergence of the ratio.
\end{proof}

\begin{remark}
The mappings
\begin{equation*}
\varphi_{\infty}\colon \xi \mapsto \xi^{\sigma} \cdot \left( \prod_{u} \widehat{\theta}_{\infty}(A_{u}x^{u}) \right)^{-1}
\text{ and }
\varphi\colon \xi \mapsto \xi^{\sigma} \cdot \left( \prod_{u} \widehat{\theta}(A_{u}x^{u}) \right)^{-1}
\end{equation*}
(see Construction~\ref{construction:frobenius}) together with the
connections
\begin{equation*}
\nabla = \widehat{\theta}^{-1} \circ d \circ \widehat{\theta}
\text{ and }
\nabla_{\infty} = \widehat{\theta}_{\infty}^{-1} \circ d \circ \widehat{\theta}_{\infty}
\end{equation*}
define two overconvergent unit-root F-isocrystal structures on
\begin{equation*}
B = K\langle x_{1},\ldots,x_{n+r} \rangle^{\dagger}.
\end{equation*}
The above argument shows that the mapping \(\rho\) induces an
isomorphism between \((B,\nabla,\varphi) = g^{\ast}\mathcal{L}_{\pi}\) and
\((B,\nabla_{\infty}, \varphi_{\infty})\).  What we need is the slightly
stronger result, namely \(\beta_{1}\) also preserves the subspaces
\(B_{I}\) of \(B\): this is to facilitate the chain level argument.
\end{remark}

The plan now is to examine traces and Fredholm determinants of the
operator \(\beta_{a}\colon B_{I} \to B_{I}\) for subsets \(I\) of
\(\{1,2,\ldots,n+r\}\), under the ``basis''
\[
\{x_{1}^{u_{1}}\cdots x_{n+r}^{u_{n+r}}:u_{j}\geq0\text{ and }u_{i}\geq 1\text{ if }i\in I\}
\]
of \(B_{I}\).  Adolphson and Sperber's idea is to use some
combinatorial quantity to measure the absolute value of the entries of the
matrix representation of \(\beta_{a}\), which we introduce below.

\begin{definition}
\label{definition:space-biprime}
Let \(\Delta\) be the Newton polyhedron of \(g\) at infinity, that is,
the convex closure of \(0 \in \mathbb{R}^{n+r}\) and
\(\{u \in \mathbb{N}^{n+r} : \text{the coefficient of }x^u\text{ in }
g \text{ is nonzero}\}\).  Let \(C(\Delta)\) be the smallest conical
region spanned by \(\Delta\).

Define the weight function
\(w\colon \mathbb{R}^{n+r}\to \mathbb{R}\) by
\[w(y_1,\ldots, y_{n+r}) =y_{n+1}+\cdots + y_{n+r},\] and define
\begin{equation*}
w_I = \min\{w(y) : y \in C(\Delta) \cap \mathbb{N}^{n+r}, y_i > 0, \forall i \in I\}.
\end{equation*}
By construction, if
\(g(x) = \sum_{u\in \mathbb{N}^{n+r}} a_{u}x^{u}\), then
\(a_{u} \neq 0\) implies that \(u \in \Delta \cap \mathbb{N}^{n+r}\).
Since \(\Phi(x) = \prod_{u} \theta_{\infty}(A_{u}x^{u})\) (see
Construction~\ref{construction:better-dwork-operator}), the subscripts
of the nonzero coefficients of \(\Phi\) in the power series expansion
are all non-negative integral linear combinations of \(u \in \Delta\).
Hence, if we write
\(\Phi(x) = \sum_{u\in \mathbb{N}^{n+r}} \Phi_{u} x^{u}\), then
\(\Phi_{u} \neq 0\) implies \(u \in C(\Delta)\).
\end{definition}

The following theorem is due to
Adolphson--Sperber~\cite[Proposition~4.2]{adolphson-sperber:chevalley-warning}.
We have only formulated a weaker version on estimating the first slope
which is sufficient for the purpose of this paper.

\begin{theorem}[Adolphson--Sperber]
\label{theorem:adolphson-sperber}
For any reciprocal root \(\lambda\) of the Fredholm
determinant \(\det(1-t\beta_{a}| B_{I})\), we have
\begin{equation*}
\mathrm{ord}_{q}(\lambda) \geq w_{I}.
\end{equation*}
\end{theorem}

\begin{proof}
Recall \(\mathbb{F}_{q} = \mathbb{F}_{p^{a}}\).  By a standard
argument
(see~\cite[Lemma~7.1]{dwork:zeta-function-hypersurface-2}, or \cite[Eq.~(46)]{bombieri:exponential-sums-in-finite-fields-1}), a point
\((x,y) \in \mathbb{R}^{2}\) is a vertex of the Newton polygon of
\(\det(1-t\beta_{a}| B_{I})\) computed with respect to
\(\mathrm{ord}_{q}\) if and only if \((ax,ay)\) is a vertex of the
Newton polygon of \(\det(1-t\beta_{1}| B_{I})\) (view \(\beta_{1}\)
as a \(\mathbb{Q}_{p}(\zeta_{p})\)-linear operator) with respect to
the valuation \(\mathrm{ord}_{p}\).  Hence, it suffices to estimate
the smallest slope of the \(p\)-adic Newton polygon of
\(\det(1 - t\beta_{1}| B_{I})\).

Recall that \(\beta_{1} = \tau^{-1}\circ\psi\circ \Phi\).
We write \(\Phi(x) = \sum_{u\in \mathbb{N}^{n+r}}\Phi_{u} x^{u}\).
Then
by~\cite[Eq.~(2.12)]{adolphson-sperber:newton-polyhedra-degree-l-function},
we have
\begin{equation}\label{eq:as-bound}
\operatorname{ord}_{p}(\Phi_{u}) \geq \frac{1}{p-1}w(u).
\end{equation}
(If we use \(F\) instead of \(\Phi\), we would get a worse bound.
So it is crucial to work with the Artin--Hasse representation here.)
A \(\mathbb{Q}_{p}(\zeta_{p})\)-``basis'' of \(B_{I}\) is of the form
\begin{equation*}
\lambda_{l} \cdot x^{u}, \quad l=1,\ldots,a, u \in \mathbb{N}^{n+r}, u_{i} \geq 1 \text{ if } i\in I,
\end{equation*}
where \(\{\lambda_{l}:l=1,\ldots,a\}\) is an integral basis of
\(\mathcal{O}_{K}\) over \(\mathbb{Z}_{p}[\zeta_{p}]\).  When
\(I = \emptyset\), the operator \(\beta_{1}\) is represented by an
infinite matrix
\begin{equation*}
C=
\begin{bmatrix}
  C_{00} & C_{01} & \cdots & C_{0i} & \cdots \\
  C_{10} & C_{11} & \cdots & C_{1i} & \cdots \\
  \vdots & \vdots &       & \vdots \\
  C_{i0} & C_{i1} & \cdots & C_{ii} & \cdots \\
  \vdots & \vdots &      & \vdots
\end{bmatrix},
\end{equation*}
where each \(C_{ij}\) is a finite block matrix, corresponding the
components of \(\beta_{1}(\lambda_{l}x^{v})\) with respect
to \(\lambda_{m}x^{u}\), with \(w(v) = i\), \(w(u)=j\).  Any entry in
this block matrix is of the form
\(\lambda_{m}\lambda_{l}^{-1}\Phi_{pu-v}\).  When \(I\neq \emptyset\),
the matrix representation of \(\alpha_{1}|_{B_{I}}\) is a
suitable submatrix \(C_{I}\) of \(C\).

An entry of \(C_{I}\) on the diagonal is of the form
\(\Phi_{(p-1)u}\), where \(u_{i}\geq 1\), for all \(i\in I\).
By~\eqref{eq:as-bound}, we find that the trace of
\(\beta_{1}|_{B_{I}}\) satisfies the estimate
\begin{equation*}
\operatorname{ord}_{p}\operatorname{Tr}(\beta_{1}| B_{I}) \geq w_{I}.
\end{equation*}
This shows that in the power series expansion of
\(\det(1-t\beta_{1}| B_{I})\), the \(p\)-order of the coefficient
of \(t\) is at least \(w_{I}\).  When it comes to the coefficient of
\(t^{m}\) in the Fredholm determinant, where \(m\geq 2\), these are
computed as sums of \(m\times m\) principal minors of the matrix
\(C_{I}\).  By employing the determinant formula in terms of the
matrix entries, one can readily observe that this coefficient has
\(p\)-order \(\geq mw_{I}\).  The theorem can be deduced by examining
the Newton polygon of \(\det(1-t\beta_{1}| B_{I})\).
\end{proof}


\begin{proof}[End of proof of Lemma~\ref{lemma:fredholm-determinant-of-alpha-on-bi}]
Since the problem only concerns the Fredholm determinant, we can
replace \(\alpha_{a}\) by \(\beta_{a}\) by virtue of
Lemma~\ref{lemma:same-fredholm-determinant}.  By
Theorem~\ref{theorem:adolphson-sperber}, it suffices to show
\[
w_I \geq \frac{1}{d}\left( |I^{\prime}|+ \sum_{i\in I^{\prime\prime}} (d-d_i)  \right).
\]
This is a problem of linear programming.  We are dealing with
non-negative integers \(y_1,\ldots,y_{n+r}\) subject to the following
constraints:
\begin{equation}\label{eq:linear-prog}
\begin{cases}
  y_i \geq 1, \forall i \in I, \\
  y_1+\cdots + y_n\leq d_1y_{n+1}+\cdots +d_r y_{n+r},
\end{cases}
\end{equation}
and we want to control the minimum of \(y_{n+1} + \cdots + y_{n+r}\).
Write
\[
\xi_i =
\begin{cases}
  y_i - 1, & i \in I; \\
  y_i, & i \notin I.
\end{cases}
\]
Then \eqref{eq:linear-prog} is equivalent to
\begin{equation}\tag{\theequation${}^{\prime}$}
\label{eq:linear-prog-prime}
\begin{cases}
  \xi_i \geq 0, \forall i =1,2,\ldots,n+r; \\
  \sum_{i=1}^{r}d_i \xi_i \geq \sum_{i=1}^{n}\xi_i + |I^{\prime}| - \sum_{i\in I^{\prime\prime}}d_{i}
\end{cases}
\end{equation}
It follows that
\begin{align*}
  &y_{n+1} + \cdots + y_{n+r}  = \sum_{i=1}^{r} \xi_{i+r} + |I^{\prime\prime}| \\
  =& \frac{1}{d_1}\left( \sum_{i=1}^{r} (d_1-d_i)\xi_{i+r} + \sum_{i=1}^{r} d_i \xi_{i+r}\right) + |I^{\prime\prime}| \\
  \text{[apply \eqref{eq:linear-prog-prime}]}~\geq & \frac{1}{d_1}\left(|I^{\prime}| - \sum_{i\in I^{\prime\prime}}d_i\right) + |I^{\prime\prime}|
                                                     = \frac{1}{d_1}\left(|I^{\prime}| + \sum_{i\in I^{\prime\prime}}(d_1-d_i)\right).
\end{align*}
This completes the proof.
\end{proof}

\section{Two lemmas}
\label{sec:algebra}

The goal of this section is to prove two lemmas,
Lemma~\ref{lemma:new-sequence}, and Lemma~\ref{lemma:local-cohomology}
that will be used in the proofs of
Theorem~\ref{theorem:beyond-middle-dimension} and
Theorem~\ref{theorem:before-middle-dimension}.

\subsection*{An algebraic lemma}
Throughout this subsection, we assume that \(k\) is an algebraically closed field.
The common zero locus, in \(\mathbb{A}_{k}^n\), of a collection of polynomials
\(f_1,\ldots,f_r \in k[x_1,\ldots,x_n]\) will be denoted by
\(Z(f_1,\ldots,f_r)\).


\begin{lemma}%
\label{lemma:recombination}
Let \(g_{1},\ldots,g_{m}\) and \(f_{m+1},\ldots,f_{r}\) be elements
of \(k[x_{1},\ldots,x_{n}]\).  Assume that
\(\dim Z(g_1,\ldots,g_m) = n-m\), and
\begin{equation}
\label{eq:inductive-hypo}
\dim Z(g_1,\ldots, g_m) > \dim Z(g_1,\ldots, g_m, f_{m+1}, \ldots, f_r).
\end{equation}
Then, there are constants \(c_{m+1}, \ldots, c_r\) in \(k\), not all zero, such that
\begin{itemize}
\item if \(f_{m+1}\) does not vanish on \(Z(g_1,\ldots,g_m)\), \(c_{m+1}=1\),
\item if \(f_{m+1}\) vanishes identically on \(Z(g_1,\ldots,g_m)\),
\(c_{m+1}=0\), and
\item \(\dim Z(g_1,\ldots, g_m, c_{m+1}f_{m+1}+ \cdots + c_rf_r) = n - m -1\).
\end{itemize}
\end{lemma}

\begin{proof}
  Denote the irreducible components of the
  variety \(Z(g_1, ..., g_m)\) by
  \(D_1, \ldots, D_h\). These components
  have the same dimension by the unmixedness theorem.
  If \(f_{m+1}\) vanishes (identically) on all components \(D_1, \ldots, D_h\),
  then \(\dim Z(g_1,..., g_m, f_{m+1}) = \dim Z(g_1,..., g_m)\).
  In this case, we set \(c_{m+1}=0\), drop \(f_{m+1}\) from our list, and consider
  the shorter list \(\{g_1,\ldots, g_m, f_{m+2}, \ldots, f_r\}\) which
  still satisfies the condition \eqref{eq:inductive-hypo}.

  Hence, without loss of generality, we may assume that
  \(f_{m+1}\) does not identically vanish on \(D_1, \ldots, D_{h_1}\)
  with \(h_1>0\), but vanishes on \(D_{h_1+1}, \ldots, D_h\).
  If \(h_1=h\), then
  \[
    \dim Z(g_1,\ldots, g_m, f_{m+1}) < \dim Z(g_1,\ldots, g_m).
  \]
  Since the variety \(Z(g_1,\ldots,,g_{m},f_{m+1})\) is non-empty, this forces
  that
  \[
    \dim Z(g_1,\ldots, g_m, f_{m+1}) = \dim Z(g_1,\ldots, g_m) - 1,
  \]
  that is, \(\dim Z(g_1, ..., g_m, f_{m+1}) = n-m-1\).

  Assume now \(h_1 < h\). By condition \eqref{eq:inductive-hypo}, there is
  another polynomial among \(\{f_{m+2}, \ldots, f_r\}\),
  say \(f_{m+2}\), which does not vanish identically on
  all of \(D_{h_1+1}, \ldots, D_h\). Without loss of generality,
  we may assume that \(f_{m+2}\) dos not vanish on
  \(D_{h_1+1}\), \(\ldots\), \(D_{h_1+h_2}\) with \(h_2>0\),
  but vanishes on \(D_{h_1+h_2+1}, \ldots, D_h\).

  \begin{claim} \textit{There is a non-zero constant \(c\) in \(k\) such that
  \(f_{m+1}+cf_{m+2}\) does not vanish on}
  \(D_1, \ldots, D_{h_1}, \ldots, D_{h_1+h_2}\).
  \end{claim}

  \begin{proof}[Proof of Claim]
    Because \(f_{m+1}\) vanishes on
    \(D_{h_1+1},\ldots, D_{h_1+h_2}\), and \(f_{m+2}\) does not,
    for any non-zero constant \(c\) in \(k\),
    the polynomial \(f_{m+1} + cf_{m+2}\) does not vanish on
    \(D_{h_1+1},\ldots, D_{h_1+h_2}\). For each \(i =1, \ldots, h_1\),
    we can choose \(x_i\) in \(D_i\) such that \(f_{m+1}(x_i)\) is non-zero as
    \(f_{m+1}\) is not identically zero on \(D_i\). Choose non-zero
    constant \(c\) in \(k\) such that none of the \(h_1\) numbers
    \begin{equation*}
      f_{m+1}(x_i) + c f_{m+2}(x_i) , i=1,\ldots, h_1
    \end{equation*}
    is zero: one simply chooses
    any non-zero \(c\) in \(k\) such that \(c\) is not among the \(h_1\) numbers
    \[\{-f_{m+1}(x_i)/f_{m+2}(x_i), i=1,\ldots, h_1\},\] which is possible
    since \(k\) is an infinite field. The claim is proved.
  \end{proof}

  Repeating the above procedure, we see
  there are constants \(c_{m+1},\ldots, c_r\) in \(k\) such
  that the linear combination
  \[
    g_{m+1} = c_{m+1}f_{m+1}+ c_{m+2}f_{m+2}+ \cdots + c_r f_r
  \]
  does not vanish identically on the component \(D_i\) for \(i=1,\ldots, h\). It
  follows that
  \[
  \dim Z(g_1,\ldots,g_m,g_{m+1}) = n-m-1.\qedhere
  \]
\end{proof}

\begin{lemma}%
  \label{lemma:new-sequence}
  Let \(f_1,\ldots, f_r \in k[x_1,\ldots,x_n]\) be a collection of polynomials.
  Set \(d_i = \deg f_i\). Assume that \(d_1 \geq d_2\geq \cdots \geq d_r\).
  Let \(Z = Z(f_1,\ldots,f_r)\).
  Then there exists a new sequence of polynomials
  \(g_1,\ldots,g_r \in k[x_1,\ldots,x_n]\) such that
  \begin{enumerate}
  \item \(Z(g_1,\ldots,g_r) = Z\),
  \item \(\deg g_i \leq d_i\),
  \item \(\dim Z(g_1,\ldots,g_{n-\dim Z}) = \dim Z\).
  \end{enumerate}
\end{lemma}

\begin{proof}
  Applying Lemma~\ref{lemma:recombination}
  repeatedly gives rise to a new sequence of polynomials
  \(g_1,\ldots,g_r \in k[x_1,\ldots,x_n]\), satisfying the following:
  \begin{itemize}
  \item
    \(g_1 = f_1\),
    \(g_m = f_m\) if \(m > n-\dim Z\);
  \item there exists an upper-triangular square matrix
    \(B = (b_{\alpha\beta})_{1\leq \alpha,\beta\leq r}\) with entries in \(k\),
    whose diagonal entries are either \(0\) or \(1\), such that
    \begin{equation}\label{eq:linear-relation}
      \begin{bmatrix}
        g_1 \\ \vdots \\ g_r
      \end{bmatrix}
      = B \cdot
      \begin{bmatrix}
        f_1 \\ \vdots \\ f_r
      \end{bmatrix};
    \end{equation}
    and
  \item \(\dim Z(g_1,\ldots,g_{n-\dim Z})=\dim Z\).
  \end{itemize}
  Thus the condition (3) is ensured.
  By construction, \(\deg g_i \leq d_i\) for any \(i=1,2,\ldots,r\). The
  condition (2) is checked.

  Proof of (1).
  Since \(g_1,\ldots,g_r\) are \(k\)-linear combinations of \(f_1,\ldots,f_r\),
  \(Z\) is contained in \(Z(g_1,\ldots,g_r)\). Let us prove
  \(Z(g_1,\ldots,g_r) \subset Z\).

  If the \(j\)\textsuperscript{th} diagonal entry of \(B\) is zero, we say \(j\) is a
  ``jumping'' index. By Lemma~\ref{lemma:recombination}, for each jumping \(j\),
  \(f_{j}\) vanishes identically on \(Z(g_1,\ldots,g_{j-1})\); hence \(f_j\)
  vanishes identically on \(Z(g_1,\ldots,g_r)\) as well.

  It remains to show that if \(\beta\) is not a jumping index,
  \(f_{\beta}(Q)=0\) for any \(Q \in Z(g_1,\ldots,g_r)\).
  For each jumping \(j\), remove the \(j\)\textsuperscript{th}
  row and \(j\)\textsuperscript{th} column from the matrix \(B\). The resulting
  matrix \(C\) is upper triangular, and its diagonal entries are all \(1\). In
  particular, \(C\) is invertible. Evaluating \eqref{eq:linear-relation} at
  \(Q\), using the vanishing of jumping \(f_j\) at \(Q\), we see that, for any
  non-jumping index \(\alpha\), we have
  \begin{equation*}
    0 = g_{\alpha}(Q) = \sum_{\beta \text{ non-jumping}} b_{\alpha\beta}f_{\beta}(Q),
  \end{equation*}
  The matrix associated with the above system of linear equations is the
  invertible matrix \(C\). Thus \(f_{\beta}(Q)=0\) for any non-jumping \(\beta\).
  This concludes the proof.
\end{proof}

\subsection*{A lemma on local cohomology}
In this subsection we prove a lemma
(Lemma~\ref{lemma:local-cohomology}) on local cohomology in the theory
of arithmetic \(\mathscr{D}\)-modules.  It is the analogue of the fact
that the Verdier dual of \(\mathbb{Q}_{\ell}[\dim Z]\) (hence
\(\mathbb{Q}_{\ell}[\dim Z]\) itself) is a perverse sheaf if \(Z\) is
a local complete intersection.

\medskip%
Let us review some basic concepts about arithmetic
\(\mathscr{D}\)-modules.  The reader is referred to Abe and
Caro~\cite[\S1]{abe-caro:theory-of-weights-in-p-adic-cohomology} or
Abe~\cite[\S1]{abe:langlands-correspondence-for-isocrystals-and-crystalline-companions}
for up-to-date surveys.

\medskip%
Let \(k\) be a perfect field (for us, \(k = \mathbb{F}_{q}\)).  Let
\(\mathcal{O}_{K}\) be a complete discrete valuation ring with residue
field \(k\) and field of fractions \(K\).  Assume that \(K\) has
characteristic \(0\).  Let
\(S = \operatorname{Spf}(\mathcal{O}_{K})\).

\medskip%
Let \(\mathcal{P}\) be a smooth formal scheme over \(S\).  Then
Berthelot~\cite[\S2.4]{berthelot:arithmetic-d-modules-1-differential-operators-of-finite-level}
introduced a sheaf \(\mathscr{D}^{\dagger}_{\mathcal{P},\mathbb{Q}}\)
on \(\mathcal{P}\) whose sections are infinite order differential
operators on \(\mathcal{P}\) of \emph{finite level}.  Caro introduced
several finiteness conditions on coherent
\(\mathscr{D}^{\dagger}_{\mathcal{P},\mathbb{Q}}\)-modules:
overcoherence
\cite[Définition~3.1.1]{caro:overcoherent-arithmetic-d-modules},
overholonomicity
\cite[Définition~3.1]{caro:overholonomic-arithmetic-d-modules}, and
``devissability'' by overconvergent F-isocrystals
\cite[Définition~3.2.5]{caro:overconvergent-f-isocrystals-and-differential-overcoherence},
\cite[Definition~2.3.1]{caro-tsuzuki:overholonomicity-of-overconvergent-f-isocrystals}.
Finally, in
\cite[Theorem~2.3.16]{caro-tsuzuki:overholonomicity-of-overconvergent-f-isocrystals},
Caro and Tsuzuki proved that with the presence of Frobenius
structures, these finiteness notions are equivalent, agreeing with
holonomicity introduced by Berthelot~\cite[\S3]{berthelot:introduction-to-arithmetic-d-modules}.

\medskip%
Let \(\mathrm{Hol}(\mathcal{P})\) denote the strictly full, thick
subcategory of
\(\mathscr{D}^{\dagger}_{\mathcal{P},\mathbb{Q}}\)-modules generated
by holonomic
\(\mathscr{D}^{\dagger}_{\mathcal{P},\mathbb{Q}}\)-modules that can be
endowed with \(q^{s}\)-power Frobenius structures for some \(s\),
although the Frobenius structure is not part of the defining data.
Let \(D^{b}_{\mathrm{hol}}(\mathcal{P})\) be the strictly full
subcategory of
\(D^{b}(\mathscr{D}^{\dagger}_{\mathcal{P},\mathbb{Q}})\) consisting
of complexes with cohomology lying in \(\mathrm{Hol}(\mathcal{P})\).

Since \(D^{b}_{\mathrm{hol}}(\mathcal{P})\) is stable under the usual
and extraordinary direct image functors, the usual and extraordinary
inverse image functors, and duality functors, as shown in
\cite{caro:overholonomic-arithmetic-d-modules}, one can use these
categories to canonically associate to each realizable \(k\)-variety
(defined below) \(V\) a coefficient category
\(D^{b}_{\mathrm{hol}}(V/K)\) for \(p\)-adic cohomology theory,
amenable of the Grothendieck six functor formalism, see
\cite{caro:6op,abe-caro:theory-of-weights-in-p-adic-cohomology}.

\medskip%
A \emph{realizable variety} \(V\) over \(k\) is a variety that admits
an immersion \(V \to \mathcal{P}\), where \(\mathcal{P}\) is a smooth,
\emph{proper} formal scheme over \(S\).  For each realizable variety
\(V\) over \(k\), and any immersion \(V \to \mathcal{P}\) as above, we
have a functor
\(D^{b}_{\mathrm{hol}}(V/K) \to D^{b}_{\mathrm{hol}}(\mathcal{P})\),
and this functor induces an equivalence between
\(D^{b}_{\mathrm{hol}}(V/K)\) and the strictly full subcategory of
\(D^{b}_{\mathrm{hol}}(\mathcal{P})\) consisting of objects that are
supported on \(V\), in sense that there is an isomorphism
\begin{equation*}
\mathcal{M} \xrightarrow{\sim} \mathbb{R}\underline{\Gamma}_{\overline{V}}^{\dagger} (\mathcal{M} ({}^{\dagger}(\overline{V}\setminus V))),
\end{equation*}
where \(\overline{V}\) is the Zariski closure of \(V\) in \(\mathcal{P} \otimes_{\mathcal{O}_{K}}k\).
For the definition of the local cohomology functor
\(\mathbb{R}\underline{\Gamma}^{\dagger}_{Z}\), see
\cite[2.1.3 (divsior case) and Définition~2.2.6 (general case)]{caro:overcoherent-arithmetic-d-modules}.
For the definition of the functor
\(\mathcal{M} \mapsto \mathcal{M}({}^{\dagger}Z)\), see
\cite[Définition~2.2.6]{caro:overcoherent-arithmetic-d-modules}.

\medskip%
For any realizable variety \(V\) over \(k\),
\(D^{b}_{\mathrm{hol}}(V/K)\) has a standard t-structure
(\cite[\S1.2]{abe-caro:theory-of-weights-in-p-adic-cohomology}).
The objects in the heart of \(D^{b}_{\mathrm{hol}}(V/K)\) are analogues to
perverse sheaves on \(V\) in the \(\ell\)-adic theory.

\medskip%
For a morphism \(f\colon V \to V^{\prime}\) of
realizable varieties, we have the ordinary and extraordinary inverse image
functors
\begin{equation*}
f^{+}, f^{!} \colon D^{b}_{\mathrm{hol}}(V^{\prime}/K) \to D^{b}_{\mathrm{hol}}(V/K),
\end{equation*}
and the ordinary and extraordinary direct image functors
\begin{equation*}
f_{+}, f_{!} \colon D^{b}_{\mathrm{hol}}(V/K) \to D^{b}_{\mathrm{hol}}(V^{\prime}/K).
\end{equation*}
They satisfy the usual adjunction properties.  For each \(V\), there
is a duality functor
\(\mathbb{D}_{V}\colon D^{b}_{\mathrm{hol}}(V/K)^{\mathrm{op}} \to
D^{b}_{\mathrm{hol}}(V/K)\), which is a t-exact anti-equivalence
satisfying \(\mathbb{D}_{V}^{2} = \mathrm{Id}\).  The duality functors
swap ordinary and extraordinary direct and inverse images:
\begin{align*}
\mathbb{D}_{V^{\prime}} f_{!} = f_{+}\mathbb{D}_{V}, \quad&
\mathbb{D}_{V^{\prime}} f_{+} = f_{!}\mathbb{D}_{V}, \\
f^{!}\mathbb{D}_{V^{\prime}} = \mathbb{D}_{V}f^{+}, \quad&
f^{+}\mathbb{D}_{V^{\prime}} = \mathbb{D}_{V}f^{!}.
\end{align*}

\begin{example}\label{example:structure-sheaf-of-an}
Suppose \(\mathcal{P}\) is a purely \(n\)-dimensional, proper smooth
formal scheme over \(S\).  Let \(H\) be a divisor of
\(P = \mathcal{P}\otimes_{\mathcal{O}_{K}}k\) and
\(U = P \setminus H\).  Then we have an equivalence
\begin{equation*}
D^{b}_{\mathrm{hol}}(U/K) \simeq \{\mathcal{M} \in D^{b}_{\mathrm{hol}}(\mathcal{P}) : \mathcal{M} \xrightarrow{\sim} \mathcal{M}({}^{\dagger}H)\}.
\end{equation*}
Under this equivalence, the complex \(K_{U}[n] = a^{+}K[n]\), where
\(a\colon U \to \operatorname{Spec}k\) is the canonical morphism, is
represented by
\(\mathcal{O}_{\mathcal{P},\mathbb{Q}}({}^{\dagger}H)\), the
(specialization of the) sheaf of function on the rigid analytic space
\(\mathcal{P}_{K}\) overconvergent along \(H\).  Moreover,
\(\mathcal{O}_{\mathcal{P},\mathbb{Q}}({}^{\dagger}H)\) is self-dual,
that is, \(\mathbb{D}_{U}(K_{U}[n]) \simeq K_{U}[n]\).  See
\cite[\S1.5.6]{abe:langlands-correspondence-for-isocrystals-and-crystalline-companions}.
Lastly, we mention that the objects in the heart of the t-structure of
\(D^{b}_{\mathrm{hol}}(U/K)\) are represented by actual holonomic
modules with overconvergent singularities along \(H\):
\begin{equation*}
\{ \mathcal{M} \in \mathrm{Hol}(\mathcal{P}) : \mathcal{M} \xrightarrow{\sim} \mathcal{M}({}^{\dagger}H)\}.
\end{equation*}
\end{example}

\begin{lemma}\label{lemma:local-cohomology}
Regard \(\mathbb{A}^{N}_{k}\) as a locally closed subscheme of the
formal projective space \(\widehat{\mathbb{P}}^{N}\) over
\(\operatorname{Spf}(\mathcal{O}_{K})\).  Let \(f_{1},\ldots, f_{r}\)
be regular functions on \(\mathbb{A}^{N}_{k}\), defining a closed
subscheme \(Z\) of \(\mathbb{A}^{N}_{k}\).  Assume that
\(\dim Z = N-r\), then the local cohomology
\(\mathscr{D}^{\dagger}_{\widehat{\mathbb{P}}^{N},\mathbb{Q}}\)-modules
\(\mathcal{H}^{m}\{\mathbb{R}\underline{\Gamma}_{Z}^{\dagger}(\mathcal{O}_{\widehat{\mathbb{P}}^{N},\mathbb{Q}}({}^{\dagger}H))[r]\}\)
are zero unless \(m = 0\).
\end{lemma}

\begin{proof}
Consider the function
\(g\colon \mathbb{A}^{N}\times \mathbb{A}^{r} \to \mathbb{A}^{1}\)
defined by \(g = \sum x_{N+i}f_{i}\).  Let \(\mathcal{L}\) be the
\(\mathscr{D}^{\dagger}\)-module on \(\mathbb{A}^{N+r}\) obtained by
regarding the Dwork isocrystal as a \(\mathscr{D}^{\dagger}\)-module:
\(\mathcal{L}=\mathrm{sp}_{+}(g^{\ast}\mathcal{L}_{\pi})\)
(cf.~\cite[\S1.2.14]{abe-caro:theory-of-weights-in-p-adic-cohomology}).
Let \(\varpi\colon \mathbb{A}^{N+r} \to \mathbb{A}^{N}\) be the
projection.  Then by the theorem of Baldassarri and Berthelot
(Theorem~\ref{theorem:local-comparison}),
\begin{equation*}
\varpi_{+}(\mathcal{L}) \simeq \mathbb{R}\underline{\Gamma}^{\dagger}_{Z}(\mathcal{O}_{\widehat{\mathbb{P}}^{N},\mathbb{Q}}({}^{\dagger}H))[r].
\end{equation*}
In view of the discussion on the t-structure of
\(D^{b}_{\mathrm{hol}}(\mathbb{A}^{N}/K)\) in
Example~\ref{example:structure-sheaf-of-an}, it suffices to show
\(\mathbb{R}\underline{\Gamma}^{\dagger}_{Z}(\mathcal{O}_{\widehat{\mathbb{P}}^{N},\mathbb{Q}}({}^{\dagger}H))[r]\)
lies in the heart of the t-structure.

Since \(\varpi\) is an affine morphism, by
\cite[Proposition~1.3.3]{abe-caro:theory-of-weights-in-p-adic-cohomology},
we have
\(\mathbb{R}\underline{\Gamma}^{\dagger}_{Z}(\mathcal{O}_{\widehat{\mathbb{P}}^{N},\mathbb{Q}}({}^{\dagger}H))[r]\in D_{\mathrm{hol}}^{\leq0}(\mathbb{A}^{N}/K)\).
It remains to show that
\(\mathbb{R}\underline{\Gamma}^{\dagger}_{Z}(\mathcal{O}_{\widehat{\mathbb{P}}^{N},\mathbb{Q}}({}^{\dagger}H))[r]\in D^{\geq0}_{\mathrm{hol}}(\mathbb{A}^{N}/K)\).
Let \(i\colon Z \to \mathbb{A}^{N}\) be the inclusion morphism.  Then
\(\mathbb{R}\underline{\Gamma}^{\dagger}_{Z}(\mathcal{O}_{\widehat{\mathbb{P}}^{N},\mathbb{Q}}({}^{\dagger}H))\)
is a representation of \(i_{+}i^{!}(K_{\mathbb{A}^N}[N])\) (cf.~\cite[\S1.1.7]{abe-caro:theory-of-weights-in-p-adic-cohomology}).
Since
\(\mathbb{D}_{\mathbb{A}^{N}}(K_{\mathbb{A}^N}[N])=K_{\mathbb{A}^N}[N]\),
it suffices to show
\(i_{+}i^{+}(K_{\mathbb{A}^N}[N-r]) \in D^{\leq0}_{\mathrm{hol}}(\mathbb{A}^{N}/K)\).
Since \(i_{+}\) is an exact functor
(\cite[Proposition~1.3.2(iii)]{abe-caro:theory-of-weights-in-p-adic-cohomology}),
it suffices to show
\(i^{+}K_{\mathbb{A}^N}[N-r] \in D^{\leq0}_{\mathrm{hol}}(Z/K)\).

To prove this last inclusion, we shall use the constructible
t-structure introduced in
\cite[\S1.3]{abe:langlands-correspondence-for-isocrystals-and-crystalline-companions}.
This is a t-structure
\(({}^{c}D^{\leq0}_{\mathrm{hol}}(V/K),{}^{c}D^{\geq0}_{\mathrm{hol}}(V/K))\)
on \(D^{b}_{\mathrm{hol}}(V/K)\).  Properties of this t-structure that
are relevant to us are the following:
\begin{itemize}
\item For any closed subvariety \(W\) of \(V\), let
\(\iota_{W}\) denote the inclusion map.  Then
\({}^{c}D^{\leq0}_{\mathrm{hol}}(V/K)\) is the full subcategory of
\(D^{b}_{\mathrm{hol}}(V/K)\) consisting of \(\mathcal{M}\) satisfying
the property that for any closed subvariety \(W\),
\begin{equation*}
\mathcal{H}^{m}\iota_{W}^{+}\mathcal{M} = 0 \text{ for any }m < \dim W.
\end{equation*}
See \cite[\S1.3.1]{abe:langlands-correspondence-for-isocrystals-and-crystalline-companions}.

\item If \(V\) is \emph{nonsingular} of pure dimension \(n\).  Any
overconvergent isocrystal \(E\) which admits some \(q^{s}\)-Frobenius
structure on \(V\) determines a \(\mathscr{D}^{\dagger}\)-module
\(\mathcal{E}\).  Then \(\mathcal{E}[-n]\) (an object in
\(D^{b}_{\mathrm{hol}}(V/K)\) of this form is called a \emph{smooth
  object} in
\cite{abe:langlands-correspondence-for-isocrystals-and-crystalline-companions})
is in the heart of the constructible t-structure.  This is simply
because the constructible t-structure is obtained by gluing smooth
objects on smooth locally closed subvarieties.  See
\cite[Proposition~1.3.3]{abe:langlands-correspondence-for-isocrystals-and-crystalline-companions}.
\end{itemize}

In the first property, if we take \(Z = W\), we get
\({}^{c}D^{\leq-\dim Z}_{\mathrm{hol}}(Z/K) \subset D^{\leq0}_{\mathrm{hol}}(Z/K)\).  Therefore, it suffices to show
\(i^{+}K_{\mathbb{A}^N}[N-r]\) falls in
\({}^{c}D^{\leq-\dim Z}_{\mathrm{hol}}(Z/K)\).  Since \(K_{\mathbb{A}^N}[N]\)
comes from the trivial isocrystal on \(\mathbb{A}^{N}\),
\(K_{\mathbb{A}^N} \in {}^{c}D^{\leq0}_{\mathrm{hol}}(\mathbb{A}^{N}/K)\) by the
second property.  Since \(i^{+}\) is \(c\)-t-exact
\cite[Lemma~1.3.2]{abe:langlands-correspondence-for-isocrystals-and-crystalline-companions},
it follows that
\begin{align*}
  i^{+}K_{\mathbb{A}^N}[N][-r]
  &= i^{+}(K_{\mathbb{A}^N})[N-r] \\
  &\in {}^{c}D^{\leq0}_{\mathrm{hol}}(Y/K)[N-r] = {}^{c}D^{\leq-\dim Z}_{\mathrm{hol}}(Y/K) \\
  &\subset D^{\leq0}_{\mathrm{hol}}(Y/K).
\end{align*}
This completes the proof.
\end{proof}

\section{Divisibility of Frobenius eigenvalues}
\label{sec:proof-bound}

We return to the following situation:
\begin{notation}
\label{situation:final}
We are given a collection of
polynomials \(f_1,\ldots,f_r \in \mathbb{F}_{q}[x_1,\ldots,x_n]\), and
denote by
\[
  Z = \operatorname{Spec}\mathbb{F}_q[x_1,\ldots,x_n]/(f_1,\ldots,f_r)
\]
the vanishing scheme of \(f_1,\ldots,f_r\).
By rearranging the order, we shall assume \(d_1\geq \cdots \geq d_r\),
where \(d_i=\deg f_{i}\).
The codimension \(n - \dim Z\) of \(Z\) is denoted by \(c\).
\end{notation}

The following easy lemma should be well-known.  It shows that the
vanishing of compactly supported cohomology of \(Z\) can be controlled
by the number of defining equations of \(Z\).  So
Theorems~\ref{theorem:beyond-middle-dimension} and
\ref{theorem:before-middle-dimension} cover all nontrivial cohomology
degrees.  In its statement, \(\mathrm{H}^{i}_{c}(Z)\) could either be
\(\mathrm{H}^{i}_{\mathrm{rig},c}(Z)\) or
\(\mathrm{H}^{i}_c(Z_{\overline{\mathbb{F}}_q},\mathbb{Q}_{\ell})\).

\begin{lemma}
  \label{lemma:vanishing-by-number-equations}
  Let \(Y\) be a nonsingular affine variety of dimension \(n\).
  Let \(f_1,\ldots,f_r \in \Gamma(Y,\mathcal{O}_Y)\) be regular functions on
  \(Y\). Let \(Z\) be the common zero locus of \(f_1,\ldots,f_r\)
  in \(Y\). Then \(\mathrm{H}^{i}_c(Z)=0\) for \(i < n-r\).
\end{lemma}

\begin{proof}
  We have a long exact sequence
  \begin{equation*}
    \cdots \to \mathrm{H}^{i}_c(Y)
    \to \mathrm{H}^{i}_c(Z) \to \mathrm{H}^{i+1}_c(Y\setminus Z)
    \to \mathrm{H}^{i+1}_{c}(Y) \to \cdots.
  \end{equation*}
  If \(i \leq n-1\), then
  \(\mathrm{H}^{i}_{c}(Y)=0\) by
  smoothness of \(Y\), Poincaré duality, and Artin vanishing.
  Thus, it suffices to prove
  \(\mathrm{H}^{i}_{c}(Y\setminus Z)=0\) for \(i<n-r+1\).

  Write \(Y \setminus Z = \bigcup_{i=r}^{r} U_{i}\),
  where \(U_i = Y\setminus \{f_i=0\}\).
  Then for \(I\subset \{1,2,\ldots,r\}\),
  \(U_{I}=\bigcap_{i\in I}U_{i}\) equals \(Y\setminus\{\prod_{i\in I}f_{i}=0\}\).
  We have a Mayer--Vietoris spectral sequence
  \begin{equation*}
    E_{1}^{-a,b} = \bigoplus_{|I|=a+1}\mathrm{H}_{c}^{b}(U_{I})
    \Rightarrow \mathrm{H}^{b-a}_c(Y\setminus Z).
  \end{equation*}
  Since each \(U_{I}\) is a smooth affine variety of dimension \(n\),
  by Poincaré duality and Artin vanishing again,
  \(\mathrm{H}^{i+1}_c(U_{I}) = 0\) if \(i < n-1\).
  It follows that
  \begin{equation*}
    E^{-a,b}_{1} \neq 0 \Longrightarrow
    \begin{cases}
      b \geq n, \text{ and }\\
      a \leq r-1,
    \end{cases} \Longrightarrow b-a\geq n-r+1.
  \end{equation*}
  Ergo, \(\mathrm{H}^{i}_{c}(Y\setminus Z)=0\)
  if \(i < n-r+1\).
\end{proof}

\begin{remark}
  (a) The same argument also works for the Betti cohomology of an algebraic
  variety \(Z\) defined by the vanishing of \(r\) regular functions on a smooth
  affine variety \(Y\) over \(\mathbb{C}\).

  (b) The lemma for rigid cohomology also follows directly from
  Corollary~\ref{corollary:compact-support-frobenius}.
\end{remark}

\medskip
The remainder of this section is devoted to the proofs of
Theorems~\ref{theorem:beyond-middle-dimension} and
\ref{theorem:before-middle-dimension}.

\subsection*{Step 1: Reduction}

Recall that \(n - \dim Z\) is denoted by \(c\).
By Lemma~\ref{lemma:new-sequence}, there exists a finite extension \(k^{\prime}\) of
\(\mathbb{F}_q\), and a collection of polynomials \(g_1,\ldots,g_r\),
such that
\begin{itemize}
\item \(\deg g_1=d_1\), \(\deg g_i \leq d_i\);
\item \(Z\) is the common zero locus of \(g_1,\ldots,g_r\), and
\item
\(\operatorname{Spec}k^{\prime}[x_1,\ldots,x_n]/(g_1,\ldots,g_c)\) has
dimension equal to \(\dim Z\).
\end{itemize}

Since the conclusion of Theorem~\ref{theorem:beyond-middle-dimension}
is not sensitive to the base field, and since we have
\begin{align*}
  \nu_j(n;\deg g_1,\ldots,\deg g_r) & \geq \nu_j(n;d_1,\ldots,d_r)\\
  \epsilon_{m}(n;\deg g_1,\ldots,\deg g_r) &\geq \epsilon_{m}(n;d_1,\ldots,d_r),
\end{align*}
it suffices to prove the theorems with \(\mathbb{F}_q\) replaced by \(k^{\prime}\) and
\(f_i\) replaced by \(g_i\). Thus, it suffices to prove
Theorems~\ref{theorem:beyond-middle-dimension} and \ref{theorem:before-middle-dimension}
under the following additional hypothesis:
\begin{equation}\label{eq:hypothesis}
  \textit{The scheme \(\operatorname{Spec}\mathbb{F}_q[x_1,\ldots,x_n]/(f_1,\ldots,f_c)\) has dimension equal to \(\dim Z\).}
\end{equation}

\subsection*{Step 2: Slope estimates}
Let us first prove that the numbers \(\nu_{j}(n; d_1,\ldots,d_r)\) and
\(\epsilon_{m}(n;d_1,\ldots,d_r)\) provide
lower bounds of the \(q\)-order of the Frobenius eigenvalues of
\(\mathrm{H}^{\ast}_{\mathrm{rig},c}(Z)\). Later we will bootstrap this bound
to a bound of \(q\)-divisibility of algebraic numbers.

Recall the meaning of \(g\) \eqref{situation}, \(G\) (Construction~\ref{construction:exponential-model-complex}),
our convention of the sets \(I\), \(I^{\prime}\), and
\(I^{\prime\prime}\) made in Notation~\ref{notation},
as well as the spaces \(B_{I}\) \eqref{eq:decompose-differential-forms}.
Rewrite the
overconvergent Dwork complex \eqref{eq:dwork-complex} as the total
complex of the following double complex (in order to save ink, we have
omitted the monomials \(\mathrm{d}x^I/x^I\) in the expression):
\begin{equation}\label{eq:double-complex}
  \begin{tikzcd}[column sep=small]
    \bigoplus\limits_{\substack{|I^{\prime}|=n \\|I^{\prime\prime}|=0}}B_I \ar[r] & \cdots \ar[r] & \bigoplus\limits_{\substack{|I^{\prime}|=n \\|I^{\prime\prime}|=c}}B_I  \ar[r] & \cdots \ar[r] & \bigoplus\limits_{\substack{|I^{\prime}|=n \\|I^{\prime\prime}|=r}}B_I \\
    {\vdots} \ar[u] & {} & {\vdots} \ar[u]  & {} & {\vdots} \ar[u] \\
    \bigoplus\limits_{\substack{|I^{\prime}|=0 \\|I^{\prime\prime}|=0}}B_I \ar[r] \ar[u]& \cdots \ar[r] & \bigoplus\limits_{\substack{|I^{\prime}|=0 \\|I^{\prime\prime}|=c}}B_I \ar[u]  \ar[r] & \cdots \ar[r] & \bigoplus\limits_{\substack{|I^{\prime}|=0 \\|I^{\prime\prime}|=r}}B_I \ar[u]
  \end{tikzcd}
\end{equation}
In the diagram, the horizontal differentials are induced by
\(D^{\prime}_{n+1},\ldots,D^{\prime}_{n+r}\),
and the vertical ones are induced by \(D^{\prime}_1,\ldots,D^{\prime}_{n}\),
where
\begin{equation}
D^{\prime}_{i} = x_{i}\frac{\partial}{\partial x_i} + \pi x_{i}\frac{\partial G}{\partial x_{i}}
=\exp(-\pi G) \circ x_{i}\frac{\partial}{\partial x_i}\circ \exp(\pi G), \quad i=1,2,\ldots,n+r.
\end{equation}

The following lemma shows that the \(0\)\textsuperscript{th},
\(1\)\textsuperscript{st}, \(\ldots\), and \((c-1)\)\textsuperscript{st} columns
of the \(E_1\)-page of the spectral sequence associated to the double
complex~\eqref{eq:double-complex} are all zero.

\begin{lemma}
  \label{lemma:c-exactness-of-rows}
  For each \(0\leq i\leq n\), the \(i\)\textsuperscript{th} row of
  \eqref{eq:double-complex} is exact in cohomology degree \(0,1,\ldots,c-1\).
\end{lemma}

\begin{proof}
The complex \eqref{eq:double-complex} uses ``toric'' conventions, and
is more suitable for later chain level manipulations.  In the
following proof, we shall use an equivalent ``affine'' convention.
Write \(B = K\langle x_{1},\ldots,x_{n+r} \rangle^{\dagger}\), and let
\begin{equation*}
D_{j} = \frac{\partial}{\partial x_{j}} + \pi \frac{\partial G}{\partial x_{j}}, \quad j=1,\ldots,n+r.
\end{equation*}
Then \eqref{eq:double-complex} can be written as
\begin{equation}\tag{\ref*{eq:double-complex}\({}^{\prime}\)}
\label{eq:double-complex-prime}
\begin{tikzcd}
\vdots & \vdots & \vdots \\
B^{\oplus n} \ar{r} \ar{u} & B^{\oplus nr} \ar{u} \ar{r} & B^{\oplus n\binom{r}{2}} \ar{r} \ar{u} & \cdots \\
B \ar{r} \ar{u} & B^{\oplus r} \ar{r} \ar{u} & B^{\binom{r}{2}} \ar{r} \ar{u} & \cdots
\end{tikzcd}
\end{equation}
In this double complex, the horizontal arrows are induced by
\(D_{n+1},\ldots,D_{n+r}\), and the vertical arrows are induced by
\(D_{1},\ldots,D_{n}\).

Let us first show the zeroth row of \eqref{eq:double-complex-prime} is
exact in degrees \(0,1,\ldots,c-1\).  Since the
\(i\)\textsuperscript{th} row of \eqref{eq:double-complex-prime} is an
\(\binom{n}{i}\)-fold direct sum of the zeroth row, the desired
exactness in general will follow.

The zeroth row of \eqref{eq:double-complex-prime} is itself the total
complex of a double complex:
\begin{equation}\label{eq:smaller}
\begin{tikzcd}
\vdots & \vdots & \vdots \\
B^{\oplus \binom{r-c}{2}} \ar{r} \ar{u} & B^{\oplus c\binom{r-c}{2}} \ar{r} \ar{u} & B^{\oplus\binom{r-c}{2}\binom{c}{2}} \ar{r}\ar{u} &\cdots \\
B^{\oplus (r-c)} \ar{r} \ar{u} & B^{\oplus c(r-c)} \ar{r}\ar{u} & B^{\oplus(r-c)\binom{c}{2}} \ar{r} \ar{u} &\cdots & \\
B \ar{r} \ar{u} & B^{\oplus c} \ar{r} \ar{u} & B^{\oplus \binom{c}{2}} \ar{r} \ar{u} & \cdots
\end{tikzcd}.
\end{equation}
In the diagram, the horizontal arrows are induced by
\begin{equation*}
D_{n+1},\ldots, D_{n+c},
\end{equation*}
and the vertical arrows are induced by
\begin{equation*}
D_{n+c+1},\ldots, D_{n+r}.
\end{equation*}
Thus the \(i\)\textsuperscript{th} row is the \(\binom{r-c}{i}\)-fold
direct sum of the zeroth row.  If we proved the zeroth row of
\eqref{eq:smaller} is acyclic except in top cohomology degree, then
every row of \eqref{eq:smaller} will be exact except in top cohomology
degree.  Thus the total complex of \eqref{eq:smaller}, which is the
zeroth row of the original \eqref{eq:double-complex-prime}, will have
vanishing cohomology in degrees \(0,1,\ldots,c-1\).

\medskip%
To prove acyclicity, we make the following auxiliary construction.
Consider the projection
\[
\varpi^{\prime}\colon \mathbb{A}^{n+r} \to \mathbb{A}^{n+r-c}, \quad (x_{1}, \ldots, x_{n+r}) \mapsto (x_{1}, \ldots, x_{n}, x_{n+c+1}, \ldots, x_{n+r}).
\]
Define
\(\mathcal{P} = \widehat{\mathbb{P}}^{n+r-c}_{\mathcal{O}_{K}} \times \widehat{\mathbb{P}}^{c}_{\mathcal{O}_{K}}\)
and
\(\mathcal{P}^{\prime} = \widehat{\mathbb{P}}^{n+r-c}_{\mathcal{O}_{K}}\).
The special fiber of \(\mathcal{P}\) (resp. \(\mathcal{P}^{\prime}\))
contains \(\mathbb{A}^{n+r}\) (resp. \(\mathbb{A}^{n+r-c}\)) as a
Zariski open subset, with complement \(H\) (resp. \(H^{\prime}\)).
Let \(\mathcal{L}\) be the holonomic
\(\mathscr{D}^{\dagger}_{\mathcal{P}, \mathbb{Q}}\)-module associated
with the Dwork crystal \(g^{\prime\ast}\mathcal{L}_{\pi}\), where
\[
g^{\prime} = x_{n+1} f_{1} + \cdots + x_{n+c} f_{c}.
\]
By construction, \(\mathcal{L}\) has an overconvergent singularity
along \(H\), i.e., \(\mathcal{L} = \mathcal{L}({}^{\dagger}H)\), and
is naturally a
\(\mathscr{D}^{\dagger}_{\mathcal{P},\mathbb{Q}}({}^{\dagger}H)\)-module,
where
\(\mathscr{D}^{\dagger}_{\mathcal{P},\mathbb{Q}}({}^{\dagger}H)\) is
the sheaf of differential operators on \(\mathcal{P}\) of finite level
with overconvergent singularities along \(H\)
(see \cite[\S4.2.5]{berthelot:arithmetic-d-modules-1-differential-operators-of-finite-level} or
\cite[\S2.5]{baldassarri-berthelot:dwork-cohomology-for-singular-hypersurfaces}).

Thus, \(\mathcal{L}\) represents an object of
\(D^{b}_{\mathrm{hol}}(\mathbb{A}^{n+r}/K)\). By
Theorem~\ref{theorem:local-comparison}, we have
\begin{equation*}
\varpi^{\prime}_{+} \mathcal{L} \simeq \mathbb{R}\underline{\Gamma}_{Z^{\prime}}(\mathcal{O}_{\mathcal{P}^{\prime}, \mathbb{Q}}({}^{\dagger}H^{\prime}))[c],
\end{equation*}
where \(Z^{\prime}\) is the zero locus of \(f_{1}, \ldots, f_{c}\) in
the affine space \(\mathbb{A}^{n} \times \mathbb{A}^{r-c}\),
that is
\[
Z^{\prime} = \{x \in \mathbb{A}^{n} : f_{1}(x) = \cdots = f_{c}(x) = 0\} \times \mathbb{A}^{r-c}.
\]
Given Hypothesis \eqref{eq:hypothesis}, we have
\(\dim Z^{\prime} = n + r - 2c\), and it follows from
Lemma~\ref{lemma:local-cohomology} that
\(\mathbb{R}\underline{\Gamma}_{Z^{\prime}}(\mathcal{O}_{\mathcal{P}^{\prime}, \mathbb{Q}}({}^{\dagger}H^{\prime}))[c]\),
initially a complex of
\(\mathscr{D}^{\dagger}_{\mathcal{P}^{\prime},\mathbb{Q}}({}^{\dagger}H^{\prime})\)-modules,
is concentrated in degree \(0\) only.  Consequently, by
\cite[Theorem~2.3]{baldassarri-berthelot:dwork-cohomology-for-singular-hypersurfaces}
(which asserts that the category of
\(\mathscr{D}^{\dagger}_{\mathcal{P}^{\prime},\mathbb{Q}}({}^{\dagger}H)\)-modules
is equivalent to the category of modules over the overconvergent Weyl
algebra, the equivalence being induced by the functor
\(\mathrm{H}^{0}(\mathcal{P}^{\prime},-)\)), we find that
\[
\mathrm{H}^{0}(\mathcal{P}^{\prime}; \mathbb{R}\underline{\Gamma}_{Z^{\prime}}(\mathcal{O}_{\mathcal{P}^{\prime}, \mathbb{Q}}({}^{\dagger}H^{\prime}))[c])
\]
is a module over the overconvergent Weyl algebra
\(\mathrm{H}^{0}(\mathcal{P}^{\prime};\mathscr{D}^{\dagger}_{\mathcal{P}^{\prime},\mathbb{Q}}({}^{\dagger}H^{\prime}))\),
and is thus concentrated in degree \(0\) only.

\medskip%
Now, the zeroth row of \eqref{eq:smaller}, shifted to the left by \(r\), is
\[
{B} \to B^{\oplus c} \to B^{\oplus \binom{c}{2}} \to \cdots \to B^{\oplus \binom{c}{c-1}} \to \underset{\bullet}{B^{\oplus \binom{c}{c}}}
\]
(where the bullet indicates the degree zero item of the chain
complex).  We want to show it is acyclic except in degree \(0\).  Since this is
a relative de~Rham complex,  according to
\cite[p.~197, Remark]{baldassarri-berthelot:dwork-cohomology-for-singular-hypersurfaces},
it represents the complex
\(\mathrm{H}^{0}(\mathcal{P}^{\prime}; \varpi^{\prime}_{+}\mathcal{L})\),
which is isomorphic to
\[
\mathrm{H}^{0}(\mathcal{P}^{\prime}; \mathbb{R} \underline{\Gamma}_{Z^{\prime}} (\mathcal{O}_{\mathcal{P}^{\prime}, \mathbb{Q}} ({}^{\dagger} H^{\prime})) [c])
\]
by Theorem~\ref{theorem:local-comparison}.  As discussed in the
preceding paragraph, this latter complex is indeed concentrated in
degree \(0\).  This completes the proof.
\end{proof}

Since Lemma~\ref{lemma:c-exactness-of-rows} implies
the spectral sequence associated to the double complex
\eqref{eq:double-complex} satisfies
\begin{equation*}
  E_{1}^{i,j} = 0, \quad \forall i < c,
\end{equation*}
in Figure~\ref{figure:square-identification},
only the shaded part contributes to the final abutment of the
spectral sequence.

\begin{figure}[ht!]
  \centering
  \begin{tikzpicture}[scale=0.9]
    \shadedraw[color=gray!15!white,draw=black] (1.5,0) rectangle (4,4);
    \draw[ultra thick,red] (2,4) -- (4,2);
    \draw[thick,dashed] (1.5,4) -- (4,1.5);
    \draw[ultra thick,blue] (1.5,3) -- (4,0.5);
    \draw[->] (0,0) -- (4.5,0) node[anchor=north west] {\(|I^{\prime\prime}|\)};
    \draw[->] (0,0) -- (0,4.5) node[anchor=south east] {\(|I^{\prime}|\)};
    \draw (4,-0.5) node{\(r\)};
    \draw (5.5,2.1) node{{\footnotesize\color{red}Second case}};
    \draw (5.3,0.5) node{{\footnotesize\color{blue}First case}};

    \draw (4,4.5) node{\footnotesize\((r,n)\)};
    \draw (4,-0.5) node{\(r\)};
    \draw (1.5,-0.5) node{\(c\)};
    \shadedraw[ball color=red!15!red, draw=gray!50] (4,4) circle (0.081);
    \draw (2,4.5) node{\footnotesize\((r-j,n)\)};
    \shadedraw[ball color=red!15!red, draw=gray!50] (2,4) circle (0.081);
  \end{tikzpicture}
  \caption{The double complex}
  \label{figure:square-identification}
\end{figure}

For this reason,
\(\mathrm{H}^{n+r-j}_{\mathrm{rig}}(\mathbb{A}_{\mathbb{F}_q}^{n+r},\mathcal{L}_{\pi})\) is in fact a
subquotient of
\begin{equation}\label{eq:direct-sum}
  \bigoplus_{\substack{|I|=n+r-j \\ |I^{\prime\prime}| \geq c}} B_{I};
\end{equation}
and Lemma~\ref{lemma:factor} can be refined:

\begin{lemma}\label{lemma:refined}
In Situation~\ref{situation:final}, under Hypothesis~\ref{eq:hypothesis},
the Fredholm determinant \(\det(1-t\cdot F| \mathrm{H}^{n-r+j}_{\mathrm{rig},c}(Z))\)
is a factor of
\[
  \prod_{\substack{|I|=n+r-j \\ |I^{\prime\prime}| \geq c}}\det(1-t\cdot q^{j-r}\alpha_{a}| B_{I}).
\]
\end{lemma}

Hence every Frobenius eigenvalue of \(\mathrm{H}^{n-r+j}_{\mathrm{rig},c}(Z)\)
is a reciprocal root of \(\det(1-t\cdot q^{j-r}\alpha_{a} | B_{I})\)
for some \(I\) with \(|I|=n+r-j\), \(|I^{\prime\prime}| \geq c\).

\medskip
To proceed, there are two cases.

\subsubsection*{First case:} \(j\geq r-c\).  This case
corresponds to Theorem~\ref{theorem:beyond-middle-dimension}.  All the
relevant spaces \(B_{I}\) lie on the lower slant line of
Figure~\ref{figure:square-identification}.
Let \(\gamma\) be one reciprocal root of
\(\det(1-t\cdot q^{j-r}\alpha_{a} | B_{I})\).
By Lemma~\ref{lemma:fredholm-determinant-of-alpha-on-bi},
the \(q\)-order of any reciprocal root \(\gamma\)
of \(\det(1-t\cdot \alpha_{a}| B_I)\) is at least
\(\frac{1}{d_1}\left( n+r-j+ \sum_{i\in I^{\prime\prime}} (d_1-d_i-1) \right)\).
Since \(d_1\geq \cdots \geq d_r\), and \(|I^{\prime\prime}|\geq c\),
we have
\begin{align*}
  \sum_{i\in I^{\prime\prime}}(d-d_i-1) &\geq \sum_{i=1}^{c}(d-d_i-1)-\sum_{i=c+1}^{|I^{\prime\prime}|-c}d_{i}^{*} \\
  &\geq \sum_{i=1}^{c}(d-d_i-1)-\sum_{i=c+1}^{r}d_{i}^*,
\end{align*}
where, recall that
\begin{equation*}
  d_i^* =
  \begin{cases}
    d_i, & \text{ if }1\leq i \leq c; \\
    1, & \text{ if }i>c, \text{ and }d_i = d_1; \\
    0, & \text{ if }i> c, \text{ and }d_i < d_1.
  \end{cases}
\end{equation*}
Moreover, by \cite[Lemma~3.1]{wan:poles},
the \(q\)-order of every reciprocal root of
\(\det(1-t\cdot \alpha_{a}| B_I)\)
is at least \(|I^{\prime\prime}| \geq c\).

\begin{remark}
  It should be noted that the cited lemma was stated for a certain Banach space
  denoted by \(B^{J_1,J_2}\) in \cite{wan:poles}. In our context, its role is
  subsumed by the overconvergent space \(B_{I}\), where \(J_1=I^{\prime}\), \(J_2=I^{\prime\prime}\).
  The proof of the cited lemma only uses Dwork trace formula, which is
  applicable to \(B_{I}\) as well.

  In fact, the cited lemma actually says that the reciprocal roots of
  \(\det(1-t\cdot \alpha_{a} | B_I)\) are \emph{algebraic integers}, and are
  divisible by \(q^{|I^{\prime\prime}|}\) in the ring of algebraic integers.
\end{remark}

Hence, Lemma~\ref{lemma:refined} implies that the \(q\)-order of every Frobenius
eigenvalue of
\(\mathrm{H}^{n-r+j}_{\mathrm{rig},c}(Z)\) is at least
\begin{equation}
  \label{eq:first-case}
  j - (r - c) +
  \max\left\{ 0, \left\lceil \frac{n-j+(r-c)-\sum_{i=1}^{r}d^*_i}{d_1} \right\rceil \right\}.
\end{equation}

Making change of variable \(j -(r-c) \to j\), the above argument
implies that the \(q\)-order of Frobenius eigenvalues of
\(\mathrm{H}^{\dim Z+j}_{\mathrm{rig},c}(Z)\) are at least
\(\nu_{j}(n;d_1,\ldots,d_r)\).

\subsubsection*{Second case:} \(0\leq j < r-c\),
which corresponds to Theorem~\ref{theorem:before-middle-dimension}.
In this case, the spaces \(B_I\) appeared in \eqref{eq:direct-sum}
all lie on the upper slant
line of Figure~\ref{figure:square-identification}. Thus,
\(r\geq |I^{\prime\prime}| \geq r-j > c\).
It follows from Lemma~\ref{lemma:fredholm-determinant-of-alpha-on-bi} that
every reciprocal root \(\gamma\) of \(\det(1-t\cdot \alpha_{a}| B_I)\)
satisfies
\(\operatorname{ord}_q\gamma \geq \frac{1}{d_1}\left( n+r-j + \sum_{i\in I^{\prime\prime}} (d_1-d_i-1)\right)\).
Since \(|I^{\prime\prime}| \geq r-j\), arguing as in the first case, we have
\begin{align}
  &\frac{1}{d_1}\left( n+r-j + \sum_{i\in I^{\prime\prime}} (d_1-d_i-1)\right) \nonumber\\
  \geq &\frac{1}{d_1} \left( n+r-j + \sum_{i=1}^{r-j} (d_1-d_i-1) - \sum_{i=r-j+1}^{r}d_{i}^* \right)\nonumber\\
  =    &\frac{1}{d_1} \left( n - \sum_{i=1}^{r-j} d_i - \sum_{i=r-j+1}^{r}d_{i}^{*}\right) + r-j.
         \label{eq:second-case}
\end{align}
Again, by \cite[Lemma~3.1]{wan:poles}, we have
\(\operatorname{ord}_q\gamma \geq |I^{\prime\prime}|\geq r-j\). Hence,
by Lemma~\ref{lemma:refined},
the \(q\)-order of every Frobenius eigenvalue of
\(\mathrm{H}^{n-r+j}_{\mathrm{rig},c}(Z)\) is at least
\begin{equation*}
  \max\left\{ 0, \left\lceil \frac{n-\sum_{i=1}^{r-j}d_i-\sum_{i=r-j+1}^{r}d_i^*}{d_1} \right\rceil \right\} = \epsilon_{j}(n;d_1,\ldots,d_r).
\end{equation*}

\subsection*{Step 3. Bootstrap}
\label{bootstrap}
To finish the proof of Theorems~\ref{theorem:beyond-middle-dimension}
and \ref{theorem:before-middle-dimension}, it remains to explain why the
a~priori weaker \(q\)-order estimate given above implies the integrality as
well as divisibility in the ring of algebraic integers. All we need is this (see
Lemma~\ref{lemma:last} below):

\begin{lemma}\label{lemma:last-assertion}
  If \(\gamma\) is a reciprocal root of the
  Fredholm determinant \(\det(1 - t \alpha_{a} | B_I)\), then \(\gamma\) is
  an algebraic integer, and any Galois conjugate of it is still a reciprocal root of
  \(\det(1-t\alpha_{a}| B_I)\).
\end{lemma}

\medskip
That \(\gamma\) is an algebraic integer had been shown by the first author in
\cite[Lemma~3.1]{wan:poles}, but it will naturally come up again in the
argument below.

\medskip
Before proving Lemma~\ref{lemma:last-assertion}
let us take a tour through the Dwork theory of exponential sums.
For \(J \subset \{1,2,\ldots,n+r\}\),
let \(X_{(J)}\) be the linear subspace of \(\mathbb{A}^{n+r}\) defined by the vanishing of
the variables \((x_{j})_{j\in J}\), and \(X_{(J)}^*\) its standard embedded
torus. Let \(g_{J}\) be the restriction of
\(g\) to \(X_{(J)}\).  For the nontrivial additive character \(\Psi\)
(see~\eqref{eq:additive-character}),
consider the exponential sum
\begin{equation*}
  S_{(J),m}^* = \sum_{x\in X_{(J)}^*(\mathbb{F}_{q^m})} (\Psi\circ \mathrm{Tr}_{\mathbb{F}_{q^m}/\mathbb{F}_q})(g_J(x)).
\end{equation*}
Let \(B^J = B/\sum_{j\in J}x_{j}B\).
Applying Dwork trace formula (Theorem~\ref{theorem:dwrok-trace-formula})
to \(g_{J}\) gives
\begin{equation}\label{eq:turnoff-variable-dwork-trace}
  (q^m-1)^{n+r-|J|} \mathrm{Tr}(\alpha_{a}^{m}|{B^J}) = S_{(J),m}^*,
\end{equation}
where, as before, \(\alpha_{a}\) is the nuclear operator defined in
Section~\ref{sec:exponential-model}.

For \(I \subset \{1,2,\ldots,n+r\}\),
let \(B_{I} = (\prod_{i\in I}x_{i})\cdot B\).
Then \(B/B_I\) should be thought of as a dagger algebra lifting
the divisor \(D_{I}=\{\prod_{i\in I}x_{i}=0\}\). The divisor
\(D_I\), being of strict normal crossings, has a standard semisimplicial resolution
\begin{equation*}
  \begin{tikzcd}
    \cdots \ar[shift left=2]{r} \ar[shift right=2]{r} \ar{r}&  \bigsqcup\limits_{\substack{J \subset I \\ |J|=2}} X_{(J)}\ar[shift left]{r} \ar[shift right]{r}& \bigsqcup\limits_{\substack{J \subset I\\ |J|=1}} X_{(J)} \ar{r} & D_{I}.
  \end{tikzcd}
\end{equation*}
By inclusion-exclusion,
\(\mathrm{Tr}(\alpha_{a}^m|{B_I})=\sum_{J\subset I} (-1)^{|J|}\mathrm{Tr}(\alpha_{a}^m|{B^J})\).
Exponentiating, we get
\begin{equation}\label{eq:alternating-prod-fred}
  \det(1 - t\alpha_{a}|{B_I}) = \prod_{J\subset I} \det(1 - t\alpha_{a}|{B^J})^{(-1)^{|J|}}
\end{equation}

\begin{proof}[Proof of Lemma~\ref{lemma:last-assertion}]
  The proof of the assertion is based on a similar,
  but more precise, argument used in the visibility proof.

  For each \(J \subset \{1,2,\ldots,n+r\}\), by
  \eqref{eq:turnoff-variable-dwork-trace},
  we have
  \begin{equation*}
    \det(1-t\alpha_{a}|{B^J})^{\delta^{n+r-|J|}} = L^*_{J}(t)^{(-1)^{n+r-|J|-1}},
  \end{equation*}
  where
  \begin{equation*}
    L_{J}^*(t) = \exp\left\{ \sum_{m=1}^{\infty}S_{(J),m}^{*} \frac{t^m}{m}\right\}.
  \end{equation*}
  But by Lemma~\ref{lemma:from-l-star-to-zeta} (applying to the lower
  dimensional affine space \(X_{(J)}\)), \(L_{J}^{*}\) is an alternating product of
  zeta functions of \(\zeta_{Z^*\cap X_{(E^{\prime})}}(q^{|E^{\prime}|}t)\),
  where \(E^{\prime}\) is the intersection of a subset \(E\) of \(J\) with \(\{1,2,\ldots,r\}\).
  Thus by the second equation in Definition~\ref{definition:delta-operator},
  as well as
  \eqref{eq:alternating-prod-fred},
  \(\det(1-t\alpha_{a}|{B_I})\) is an infinite alternating product of
  ``shifted'' zeta functions
  \(\zeta_{Z^*\cap X_{(E^{\prime})}}(q^{M}t)\), where \(M \in \mathbb{N}\).

  Therefore, any reciprocal root of
  \(\det(1-t\alpha_{a}| B_I)\), say \(\gamma\), is a reciprocal
  zero or reciprocal pole of some zeta function
  \(\zeta_{Z^*\cap X_{(J)}}(q^{M}t)\), not being canceled in the infinite
  product, for some natural number \(M\).
  Therefore, \(\gamma\) is an algebraic integer.
  Since such a shifted zeta function is a ratio of integral polynomials of
  constant term one, the (reciprocal) minimal polynomial of \(\gamma\) must not be
  canceled either.
\end{proof}

By Lemma~\ref{lemma:last-assertion},
all the conjugates of \(\gamma\) are still reciprocal roots of
\(\det(1-t\alpha_{a}| B_I)\). Thus the \(q\)-orders of the conjugates are bounded by
\eqref{eq:first-case} or \eqref{eq:second-case} depending on \(|I|\).
The proof of Theorems~\ref{theorem:beyond-middle-dimension}
and \ref{theorem:before-middle-dimension} is then concluded
thanks to the following elementary lemma.

\begin{lemma}
  \label{lemma:last}
  Fix an algebraic closure \(\overline{\mathbb{Q}}_p\) of \(\mathbb{Q}_p\).
  Let \(\gamma \in \overline{\mathbb{Q}}_{p}\) be an algebraic integer. Suppose
  that for any automorphism \(\sigma\) of \(\overline{\mathbb{Q}}_{p}\),
  \(\operatorname{ord}_q(\sigma(\gamma))\geq m\). Then \(q^{m} \mid \gamma\)
  in the ring of algebraic integers.
\end{lemma}

\begin{proof}
  Let \(P(T) = T^{e} - a_{1}T^{e-1} + \cdots + (-1)^{e}a_{e}\) be the minimal
  polynomial of \(\gamma\). Then for every \(i\), \(a_{i} \in \mathbb{Z}\).
  Since \(a_i\) is an elementary symmetric polynomial of \(\sigma(\gamma)\),
  the hypothesis implies that \(\operatorname{ord}_q(a_i) \geq m\).
  Hence we can write \(a_i = q^{im} \cdot b_i\)
  for some \(b_i \in \mathbb{Z}\). The numbers \(\sigma(\gamma)\cdot q^{-m}\)
  all satisfy the  polynomial equation \(T^e-b_1T^{e-1} + \cdots + (-1)^{e}b_{e} = 0\),
  thus are all algebraic integers. The lemma is proved.
\end{proof}

\begin{proof}[Proof of Theorem~\ref{theorem:projective-bound}]
Using the long exact sequence of the pair \((\mathbb{P}^{n},Z)\),
the assertions about \(\mathbb{P}^n\setminus Z\) follow from those about \(Z\).

The previous
works~\cite{esnault-katz:cohomological-divisibility,esnault-wan:divisibility}
on \(\ell\)-adic cohomology treated affine and projective cases separately
(although using similar strategies). But in fact the projective
case is a formal consequence of the affine case.
Let \(\widehat{Z}\) be the affine cone of \(Z\). Let \(L \to \mathbb{P}^n_{\mathbb{F}_q}\) be the
geometric line bundle associated to \(\mathcal{O}_{\mathbb{P}^{n}_{\mathbb{F}_q}}(-1)\).
Then \(\widehat{Z}\setminus\{0\}\) naturally embeds into
\(L|_Z \overset{\text{def}}{=\!=} L \times_{\mathbb{P}^n_{\mathbb{F}_q}} Z\) as the
complement of the zero section. Identifying \(Z\) with the zero section of
\(L|_Z\), the relative cohomology sequence associated
to the pair \((L|_Z,\widehat{Z}\setminus\{0\})\) reads:
\begin{equation*}
  \cdots
  \to \mathrm{H}^{i}_{\mathrm{rig},c}(L|_Z) \xrightarrow{u} \mathrm{H}^{i}_{\mathrm{rig}}(Z) \to
  \mathrm{H}^{i+1}_{\mathrm{rig},c}(\widehat{Z}\setminus \{0\})
  \to \mathrm{H}^{i+1}_{\mathrm{rig},c}(L|_Z) \to \cdots.
\end{equation*}
Note that the restriction map \(u\) can be factored as
\begin{equation*}
\begin{tikzcd}
\mathrm{H}^{i}_{\mathrm{rig},c}(L|_Z) \ar[r,"u"] \ar[d] & \mathrm{H}^{i}_{\mathrm{rig}}(Z) \\
\mathrm{H}^{i}_{\mathrm{rig},c}(L) \ar[r] & \mathrm{H}^{i}_{\mathrm{rig}}(\mathbb{P}^n_{\mathbb{F}_q}) \ar[u]
\end{tikzcd}.
\end{equation*}
In the diagram, horizontal arrows are restriction maps to the zero section.
It follows that the image of \(u\) falls into the ``non-primitive part'' of the
cohomology of \(Z\). Hence the above exact sequence gives rise to a Frobenius
equivariant embedding
\(\mathrm{H}^{i}_{\mathrm{rig}}(Z)_{\mathrm{prim}} \hookrightarrow\mathrm{H}^{i+1}_{\mathrm{rig},c}(\widehat{Z} \setminus\{0\})\).
If \(i>0\), applying the relative cohomology sequence for compactly supported
cohomology, we know the map
\(\mathrm{H}^{i}_{\mathrm{rig,c}}(\widehat{Z}\setminus\{0\}) \to \mathrm{H}^{i}_{\mathrm{rig},c}(\widehat{Z})\)
is bijective. Thus for \(n-r+j > 0\), the Frobenius eigenvalues of
\(\mathrm{H}_{\mathrm{rig}}^{n-r+j}(Z)_{\mathrm{prim}}\) are also
Frobenius eigenvalues of \(\mathrm{H}^{n+1-r+j}_{\mathrm{rig,c}}(\widehat{Z})\).
We then conclude by applying Theorems~\ref{theorem:beyond-middle-dimension} and
\ref{theorem:before-middle-dimension} to \(\widehat{Z}\).
\end{proof}

\bibliographystyle{plain}
\bibliography{vi-divi}%

\end{document}